\documentclass{article}
\title{Patnaik-Pearson intrinsic dimension for internal representations of neural networks}
\date{\today}
\author{Tom Hadfield\\
    Department of Mathematics, Imperial College London\\
	180 Queen's Gate, London SW7 5HF, United Kingdom\\
	t.hadfield@imperial.ac.uk
}
\date{\today}

\usepackage{amsmath, amssymb, bm}
\usepackage{graphicx}  % in preamble
\newtheorem{theorem}{Theorem}[section]
\newtheorem{lemma}{Lemma}[section]

\newtheorem{corollary}{Corollary}[section]
\newtheorem{conjecture}{Conjecture}[section]
\newtheorem{definition}{Definition}[section]

\DeclareMathOperator*{\bbA}{\mathbb{A}}

\DeclareMathOperator*{\CC}{\mathbb{C}}

\DeclareMathOperator*{\EE}{\mathbb{E}}

\DeclareMathOperator*{\PP}{\mathbb{P}}

\DeclareMathOperator*{\RR}{\mathbb{R}}
\DeclareMathOperator*{\RV}{\mathrm{RV}}
\DeclareMathOperator*{\Cov}{\mathrm{Cov}}
\DeclareMathOperator*{\Corr}{\mathrm{Corr}}

\DeclareMathOperator*{\dist}{\mathrm{dist}}
\DeclareMathOperator*{\Pareto}{\mathrm{Pareto}}
\DeclareMathOperator*{\secondmin}{\mathrm{secondmin}}

\DeclareMathOperator*{\rto}{\rightarrow}

\DeclareMathOperator*{\CDF}{\mathrm{CDF}}

\DeclareMathOperator*{\bfxi}{{\bf x}_i}
\DeclareMathOperator*{\bfxone}{{\bf x}_1}
\DeclareMathOperator*{\bfxtwo}{{\bf x}_2}
\DeclareMathOperator*{\bfxN}{{\bf x}_N}

\DeclareMathOperator*{\bb}{{\bf b}}

\DeclareMathOperator*{\ReLU}{\mathrm{ReLU}}
\DeclareMathOperator*{\Attention}{\mathrm{Attention}}
\DeclareMathOperator*{\Rbar}{r} % redefine
\DeclareMathOperator*{\xsubi}{{\bf{x}}_i}
\DeclareMathOperator*{\xsubj}{{\bf{x}}_j}

\DeclareMathOperator*{\xsubone}{{\bf{x}}_1}

\DeclareMathOperator*{\xsubN}{{\bf{x}}_N}
\DeclareMathOperator*{\xoverline}{\overline{\bf{x}}}

\DeclareMathOperator*{\blambda}{\bm{\lambda}}
\DeclareMathOperator*{\blambdaK}{\bm{\lambda_K}}

\DeclareMathOperator*{\bphi}{\bm{\phi}}
\DeclareMathOperator*{\bphiK}{\bm{\phi_K}}
\DeclareMathOperator*{\bpsi}{\bm{\psi}}

\DeclareMathOperator*{\TwoNN}{{\mathrm{TwoNN}}}
\DeclareMathOperator*{\PatnaikPearson}{\mathrm{PP}}
\DeclareMathOperator*{\ID}{\mathrm{ID}}

\DeclareMathOperator*{\rhoemp}{\rho_{\mathrm{emp}}}
\DeclareMathOperator*{\rhotail}{\rho_{\mathrm{tail}}}

\DeclareMathOperator*{\Xcoldem}{X_{\mathrm{coldem}}}
\DeclareMathOperator*{\normXresid}{|| \Xcoldem ||}
\DeclareMathOperator*{\normalised}{\mathrm{normalised}}

\DeclareMathOperator*{\HTSR}{\mathrm{HTSR}}

\DeclareMathOperator*{\nuinfty}{\nu_\infty}
\DeclareMathOperator*{\softmax}{\mathrm{softmax}}
\DeclareMathOperator*{\argmax}{\mathrm{argmax}}
\DeclareMathOperator*{\stablerank}{\mathrm{sr}}
\DeclareMathOperator*{\diag}{\mathrm{diag}}

\DeclareMathOperator*{\Tr}{\mathrm{Tr}}
\DeclareMathOperator*{\MultiHead}{\mathrm{MultiHead}}
\DeclareMathOperator*{\LayerNorm}{\mathrm{LayerNorm}}
\DeclareMathOperator*{\FFN}{\mathrm{FFN}}
\DeclareMathOperator*{\Concat}{\mathrm{Concat}}
\DeclareMathOperator*{\head}{\mathrm{head}}
\DeclareMathOperator*{\otherwise}{\mathrm{otherwise}}
\DeclareMathOperator*{\sff}{\mathrm{sf}}
\DeclareMathOperator*{\fracEOneSqrdOverETwo}{{\frac{e_1^2}{e_2}}}
\DeclareMathOperator*{\fracTwoSigmaYoverEOne}{{\tfrac{2 \sigma(Y)}{e_1}}}
\DeclareMathOperator*{\fracSigmaYSqrdOverETwo}{{\tfrac{\sigma(Y^2)}{e_2}}}
\DeclareMathOperator*{\fracBigOOneOverD}{O({\tfrac{1}{d}})}
\DeclareMathOperator*{\formulaOne}{\fracTwoSigmaYoverEOne Z_1 - \fracSigmaYSqrdOverETwo Z_2}
\DeclareMathOperator*{\formulaTwo}{\left( \formulaOne \right)}
\DeclareMathOperator*{\formulaThree}{{\tfrac{1}{\sqrt{d}}} \formulaTwo}
\DeclareMathOperator*{\formulaFour}{1 + \formulaThree + \fracBigOOneOverD}
\DeclareMathOperator*{\formulaFive}{\fracEOneSqrdOverETwo \left[ \formulaFour \right]}
\DeclareMathOperator*{\pdka}{p(d,K,\alpha)}
\DeclareMathOperator*{\pdkab}{p(d,K,\alpha,\beta)}

\DeclareMathOperator*{\myqed}{\hfill\blacksquare}

\begin{document}
\maketitle

\begin{abstract}
We define a new measure of intrinsic dimension of a data manifold, which we call the Patnaik-Pearson dimension, and apply this to internal representations of neural networks, in particular transformers. 
The inspiration for this comes from the HTSR and SETOL work of Martin, Mahoney and Hinrichs, combined with the TwoNN intrinsic dimension estimator of Facco et al.
We prove various properties of this intrinsic dimension estimator.
Treating weight matrices of neural networks as data manifolds, for weight matrices whose empirical spectral density follows a Pareto (power law) distribution, we relate the Patnaik-Pearson dimension to the HTSR and SETOL analysis, and show that critical values of the tail exponent coincide for the two approaches.
Using a combination of theoretical and numerical techniques, including Voiculescu's free probability, we study the behaviour of the Patnaik-Pearson dimension of a data manifold under the transformations typical to neural networks.
We apply this machinery to the BERT-base and DeepSeek-R1-Distill-Qwen-1 models, to investigate first the Patnaik-Pearson dimension of the initial data manifold of token embeddings, and second the evolution of the Patnaik-Pearson dimension as token embeddings pass through the layers of the model. Code and notebooks used for the numerical results presented here is available at https://github.com/tdhadfield/PatnaikPearson 
\end{abstract}

\section{Introduction}

In this paper we develop machinery for studying
the evolution of the data manifold of internal representations of a neural network as it passes through the layers of the network, with particular application to the transformer architecture. 

Mathematical understanding of neural networks, in particular transformers \cite{attention_is_all_you_need}, is an area of intensive ongoing study by a number of mathematicians, with a variety of different approaches. 
In the work of Rigollet and coworkers \cite{math_pers_trans, mean_field_dyn}, 
the transformer is modelled as a mean-field interacting particle system on a spherical surface. The key mathematical tool here is the Wasserstein gradient flow. Tokens are particles, and each encoder layer is a discrete time-step of a continuity equation driven by a self-attention kernel. The long-time attractor is a point mass (corresponding to rank collapse and token uniformity), but there is a long-lived metastable phase of partial clustering. 
This relates to the measure-theoretic approach of Vuckovic et al. \cite{math_theory_attention}, who interpret self-attention as a system of self-interacting particles.

Neural networks have been extensively studied from the viewpoint of Riemannian geometry. 
In \cite{geometry_dln} Menon describes the training dynamics for Deep Linear Networks in terms of the geometric theory of dynamical systems, unifying results by several authors into a thermodynamic framework for deep learning.
In \cite{curved_spacetime} Di Sipio, Diaz-Rodriguez and Serrano present a geometric framework for understanding transformer-based models, via an explicit analogy to General Relativity. 

There are very active approaches via algebraic geometry.
In \cite{geom_poly_nn} Kubjas, Li and Wiesmann study the geometry of polynomial neural networks with monomial activation functions.
In \cite{algebraic_geometry_rnn} Grosdos and coworkers study rational neural networks using algebraic-geometric tools. 
When the activation function is polynomial, the corresponding function space, the ``neuromanifold", can be naturally described by polynomial equations and inequalities, enabling it to be studied algebraically.
In \cite{neuroalgebraic} Marchetti et al. build the foundations of neuroalgebraic geometry, a research direction combining algebraic geometry and deep learning. They study neural networks with polynomial activations, for which the associated function spaces are semi-algebraic varieties, and outline the correspondence between algebro-geometric invariants of these varieties and fundamental aspects of machine learning.

Topological data analysis and persistent homology provide very valuable insights.  
For example, in \cite{hidden_holes}, Fitz et al. study the evolution of internal topological structure in various LLMs across depth and time during training, finding surprising distinctions between the behaviour of transformer-based as compared to LSTM-based architectures.
In \cite{persistent_llm}, Gardinazzi and coworkers use zigzag persistence to build topological descriptors which measure how topological features develop, persist and evolve through the layers of the network.
In \cite{top_geom_dm}  Magai and Ayzenberg study internal representations of neural networks and the dynamics of changes in the topology and geometry of the data manifold on different layers.
We also mention the work of Fay et al. \cite{shape_adversarial} applying persistent homology to study how adversarial inputs to LLMs reshape the topology and geometry of their internal representation spaces.

Finally, we mention the neural operator approach of Boullé and Townsend \cite{operator_learning}. 
Operator learning aims to discover or approximate an unknown operator, which often corresponds to the solution operator associated with an unknown PDE. Neural operators generalize neural networks by taking the inputs and outputs to be functions rather than vectors.

For a very clear and comprehensive description of the mathematics underlying the transformer architecture, we recommend the paper of Noguer I Alonso \cite{complete_mot}.

In this paper we take a different approach. 
A major motivation for the work presented here is the pioneering work of Martin, Mahoney and Hinrichs \cite{spectral_rg, setol, htsr} on Heavy-Tailed Self-Regularization (HTSR) and their Semi-Empirical Theory of Learning (SETOL). 
They first of all observed empirically \cite{weight_watcher} a recurring pattern of heavy-tailed distributions of eigenvalues in the weight matrices of well-trained models, and then developed a theoretical model for this phenomenom, using techniques from statistical physics, quantum chemistry and random matrix theory.

Taking this as a starting point, we are interested in measuring the intrinsic dimension of a data manifold, and how this evolves as it is acted upon by weight matrices with specific empirical spectral densities.  
The notion of intrinsic dimension of a data manifold is not uniquely defined.
Recall that the Manifold Hypothesis \cite{testing_the_m_h, manifold_hypothesis} is the hypothesis that high dimensional data tend to lie in the vicinity of a low-dimensional submanifold.
This has been observed empirically in many real world situations, giving rise to the development of a wide range of statistical methods, and has been suggested as a key explanatory factor for the success of modern neural network architectures.
Intrinsic dimension is the dimension of this submanifold, namely the true number of degrees of freedom of the data.
In the context of internal data representation of neural networks, intrinsic dimension is the number of coordinates necessary to describe the data without meaningful information loss. It is well-known that deep neural networks are substantially over-parametrized, with significant redundancy among both weights and activations. 

There are many different approaches to defining intrinsic dimension, using a wide range of mathematical tools and capturing different aspects of the data manifold. 
A very important example for this paper is the TwoNN intrinsic dimension estimator \cite{twonn} of Facco et al. 
We also mention the L2N2 estimator of Ong et al. \cite{universal_nne}, which is proven to be universal, in the sense that it converges to the true intrinsic dimension independent of which distribution is used to generate the data.
Finally, we mention the work of Ansuini et al. \cite{int_dim_data_rep_dnn}, on intrinsic dimension of data-representations internal to neural networks, who observe that for trained networks, intrinsic dimension first increases and then progressively decreases as the data representation passes through the layers of the network.

In this work we start by studying a simple Gaussian point cloud generative model for a data manifold $X$, realised as a collection of $N$ points $\{ \bfxone, \bfxtwo, ..., \bfxN \}$ in $\RR^d$.
Our original motivation for this model came from considering the question, what is the intrinsic dimension of the data manifold represented by the token embeddings of the BERT model \cite{bert}.
We apply the TwoNN intrinsic dimension estimator \cite{twonn} to this generative model, and show that another measure of dimension naturally emerges, which we call the Patnaik-Pearson intrinsic dimension of $X$, 
due to its use of Patnaik and Pearson's moment-matching formulae \cite{patnaik_pearson, patnaik, pearson}, and denote by $\PatnaikPearson(X)$. 
We study the properties of $\PatnaikPearson(X)$ for various behaviours of our generative model. 
Treating weight matrices of neural networks as data manifolds, for weight matrices whose Empirical Spectral Density follows a Pareto (Power Law) distribution, we relate the Patnaik-Pearson dimension to the HTSR and SETOL analysis.
In particular, the critical values for the tail exponent for both the Patnaik-Pearson dimension and HTSR and SETOL coincide.

Using a combination of analytic and numerical techniques, we study the behaviour of the Patnaik-Pearson dimension of a data manifold under the transformations typical to neural networks - multiplication by weight matrices; application of activation functions and softmax; addition, interpolation and concatenation; layer normalisation; attention.
A consistent phenomenom throughout is that heavier-tailed distributions dominate lighter-tailed.
Some of these operations typically decrease Patnaik-Pearson dimension, others tend to increase it.
For a data manifold $X$, and weight matrix $W$, represented by matrices 
in general position with shapes $N \times d$ and $d \times d$ drawn from so-called Pareto ensembles, 
using Voiculescu's free probability \cite{voiculescu_limit_laws, free_random_variables} we prove that, asymptotically for large $N$, $d$, then (Theorem \ref{theorem_pp_product_pareto_ensemble}):
\begin{equation}
{\tfrac{1}{d}} \EE(\PatnaikPearson(X)) * {\tfrac{1}{d}} \EE(\PatnaikPearson(W)) \leq
{\tfrac{1}{d}} \EE(\PatnaikPearson(XW))
\leq \min \{ {\tfrac{1}{d}} \EE(\PatnaikPearson(X)), {\tfrac{1}{d}} \EE(\PatnaikPearson(W)) \}
\end{equation}

Finally, we apply this machinery to two examples - the BERT-base and DeepSeek-R1-Distill-Qwen-1 models, to study the Patnaik-Pearson dimension of the data manifold of token embeddings, and the evolution of the Patnaik-Pearson dimension as data passes through the layers of the model. Consistent with the observation of previous studies, we observe that dimension decreases.

\section{Preliminaries}

\subsection{Random Matrices and heavy-tailed distributions}

We summarise some basic material about random matrices \cite{intro_random_matrices, rmt} 
and heavy-tailed distributions \cite{evt_tsm} that we use in the sequel.\\

{\bf Singular value decomposition:} An $N \times d$ (real-valued) matrix $A$ can be decomposed as 
\begin{equation}
\label{svd}
A = U S V^T
\end{equation}
where $U$ is an orthogonal $N \times N$ matrix, $V$ is $d \times d$ orthogonal and $S$ is a diagonal $N \times d$ matrix, with diagonal entries $S_{ii} = \sqrt{\lambda_i}$, for $1 \leq i \leq \min(N,d)$,
where the $\lambda_i$ are the non-zero eigenvalues of $A A^T$ (which coincide with the non-zero eigenvalues of $A^T A$).
We impose $0 \le \lambda_1 \le \ldots \le \lambda_d$. 
This determines $S$ uniquely, and $U$, $V$ are unique up to a block-orthogonal ambiguity on each eigenspace of $A^T A$ (and simultaneous sign flips on 1-dimensional eigenspaces).\\

{\bf Marchenko-Pastur distribution:}
Suppose $W$ is an $d \times n$ random matrix, whose entries are iid $N(0, \sigma^2)$, with $\sigma^2 < \infty$.
Define
$Y = {\tfrac{1}{d}} W^T W$
which is a symmetric $n \times n$ matrix, with non-negative eigenvalues $\lambda_1, \ldots ,\lambda_n$.
Then in the limit as $n, d \rto \infty$, keeping the aspect ratio $c := n/d$ constant, the probability distribution of the eigenvalues of $Y$ 
converges to a Marchenko-Pastur distribution $MP(c, \sigma^2)$.
This is supported on $[\lambda_{-} , \lambda_{+}]$, with 
\begin{equation}
\label{mp_lambda_pm}
\lambda_{\pm} = \sigma^2 (1 \pm \sqrt{c})^2
\end{equation}
plus, if $c > 1$, an atom at 0 of mass $1 - {\tfrac{1}{c}}$.
The pdf $f(x)$ is given by 
\begin{equation}
\label{mp_pdf}
f(x) = \frac{ \sqrt{(\lambda_{+} - x)(x - \lambda_{-})} }{2 \pi \sigma^2 c x}, \quad \lambda_{-} \leq x \leq \lambda_{+}
\end{equation}
and $f(x) = 0$ otherwise, apart from the possible atom at 0.\\

{\bf Pareto distribution:} A real-valued random variable $X$ follows a Pareto (Power Law) distribution, with tail exponent $\alpha > 0$, if $\PP(X > x) = x^{-\alpha}$, for $x \geq 1$. The corresponding pdf is 
\begin{equation}
\label{pareto_alpha_pdf_X}
f(t) = {\frac{\alpha}{t^{\alpha+1}}}, \, t \geq 1, \quad f(t) = 0, \, t < 1
\end{equation}

{\bf Regularly-varying distribution:} This all follows Mikosch and Wintenberger \cite{evt_tsm}. 
A random variable $X$ (with distribution $F$) is regularly-varying, with tail exponent $\alpha > 0$, if 
$$
\PP(|X| > x) = {\frac{L(x)}{x^\alpha}}, \quad \mathrm{and} \quad {\frac{\PP(\pm X > x)}{\PP(|X| > x)}} = p_\pm
$$
where $p_{+} + p_{-} = 1$, and $L$ is slowly varying, in the sense that, for all $c > 0$
$$\lim_{x \rto \infty} {\frac{L(cx)}{L(x)}} = 1$$
We denote this property by $X \in \RV(\alpha)$.
If $X$ is Pareto with tail exponent $\alpha$, then $X \in \RV(\alpha)$. 

\begin{theorem}
\label{fellers_convolution_lemma} (Feller's convolution lemma) Assume $X \in \RV(\alpha)$, with $p_{+} > 0$, 
and either 
(i) $Y$ is independent of $X$, and $Y \in \RV(\alpha)$, or (ii) $\PP(|Y| > x) = o (\PP(|X| > x))$ as $x \rto \infty$. 
Then
\begin{equation}
\label{X_plus_Y_feller}
\PP(X + Y > x) \sim \PP(X > x) + \PP(Y > x) \quad \mathrm{as} \quad x \rto \infty
\end{equation}
\end{theorem}

\begin{corollary}
If $X \in \RV(\alpha)$, $Y \in \RV(\beta)$, then $X + Y \in \RV( \min \{ \alpha, \beta \} )$
\end{corollary}

\begin{theorem}
\label{product_rv} 
If $X, Y > 0$ are independent, with $X \in \RV(\alpha)$ for some $\alpha > 0$, and also either $Y \in \RV(\alpha)$, or $\PP(Y > x) = o(\PP(X > x))$ (e.g. if $Y \in \RV(\beta)$, for $\beta > \alpha$), then 
$X Y \in \RV(\alpha)$.
In particular, 
$$
X \in \RV(\alpha), \quad Y \in \RV(\beta), 
\quad \alpha \ne \beta \quad 
\implies 
\quad
X Y \in 
\RV( \min \{ \alpha, \beta \}
$$  
If in addition $\EE( Y^{\alpha + \epsilon} ) < \infty$ for some $\epsilon > 0$, then (Breiman's lemma):
$$\PP(X Y > x) \sim \EE(Y^\alpha) \PP(X > x)$$
\end{theorem}

{\bf Subexponential distributions:} Suppose $X$ is a non-negative random variable, and $X_i \sim X$ iid, for $i = 1,2, \ldots n$. 
Then (the distribution of) $X$ is subexponential if for all (equivalently, for some) $n \geq 2$, 
$$\lim_{x \rto \infty} {\frac{\PP( X_1 + \ldots + X_n > x)}{n \PP(X > x)}} = 1$$
This implies that $X$ obeys the ``single big jump" principle
\begin{equation}
\label{single_big_jump}
\PP( X_1 + \ldots + X_n > x) \sim \PP(M_n > x) \quad \mathrm{as} \quad x \rto \infty
\end{equation}
where $M_n = \max \{ X_1, \ldots , X_n \}$. 
By (\ref{X_plus_Y_feller}) regularly-varying implies subexponential.

\subsection{HTSR and SETOL}
\label{subsection_htsr_setol}

We give a brief summary of the Heavy-Tailed Self-Regularization (HTSR) and Semi-Empirical Theory of Learning (SETOL) work of Martin, Mahoney and Hinrichs \cite{weight_watcher, spectral_rg, setol, htsr}.

We are concerned with the weight matrices at each layer of some neural network. 
Given a real-valued $d \times n$ weight matrix $W$, define $Y$ to be the
$n \times n$ matrix
$Y ={\frac{1}{d}} {W^T}W $.
The Empirical Spectral Density (ESD) of $W$, denoted $\rhoemp(\lambda)$ is formed from the eigenvalues $\lambda_j$ of $Y$:
$$
\rhoemp(\lambda) = 
{\frac{1}{n}} 
\sum_{j=1}^n 
\delta
(\lambda - \lambda_j)
$$
For randomly initialized weights, we would expect a Marchenko-Pastur distribution (\ref{mp_lambda_pm}, \ref{mp_pdf}). 
However, numerical results consistently show that, for the best-performing pre-trained models, weight matrices (of all layers) have heavy-tailed ESDs, 
and that the tails of these ESDs 
can be fit well by a Pareto (power law) distribution, for some tail exponent $\alpha$, beyond some cutoff $\lambda_0$:
\begin{equation}
\label{htsr_alpha}
\rhotail(\lambda) = \rhoemp(\lambda) \sim \lambda^{-\alpha}, \quad \lambda \geq \lambda_0
\end{equation}
Note that this is a slightly different notational convention than used in (\ref{pareto_alpha_pdf_X}), to avoid any confusion later in this paper
note that the correspondence between the two is given by 
\begin{equation}
\label{alpha_HTSR_alpha_Pareto}
{\alpha_{\HTSR}} = {\alpha_{\Pareto}} + 1 
\end{equation}
See also Section \ref{correspondence_htsr_setol_pp_dim}.

The HTSR theory states that a neural network layer is optimally trained when its ESD can be fit to a power law distribution, with tail exponent (in the sense of (\ref{htsr_alpha})) $\alpha = 2$. 
Crucially, this seems to be a universal property of all well-trained neural networks, irrespective of training data, model architecture and training procedure.
Empirical results show that $\alpha < 2$ (the HTSR Very Heavy-Tailed class) characterises overfitting.
In addition HTSR uses random matrix theory to identify qualitatively-distinct phases of learning, 
classifying the ESD of a weight matrix $W$ into 5+1 phases of training:
\begin{enumerate}
\item{Random: corresponds to the start of training - the randomly initialized weights given rise to a Marchenko-Pastur distribution for the ESD.}
\item{Bulk plus Spikes: as training progresses, larger eigenvalues (spikes) appear and separate themselves from the Marchenko-Pastur bulk.}
\item{Weakly Heavy-Tailed: a power-law distribution for the largest eigenvalues has developed, with tail exponent $\alpha > 6$.}
\item{Heavy (Fat) Tailed: as training continues, $\alpha$ steadily decreases, and now lies in the range $2 < \alpha < 6$.}
\item{Very Heavy Tailed: $1 < \alpha < 2$. This indicates overfitting.}
\item{Rank Collapse.}
\end{enumerate}

As discussed in \cite{complete_mot, grokking_htsr}, the evolution of the ESD during training relates to the ``grokking" phenomenom \cite{grokking}, where the initial phase of model training corresponds to memorization, characterised by high intrinsic dimension of weight matrices, followed by a sharp reduction in  intrinsic dimension at the grokking transition. 
The very useful WeightWatcher package \cite{weight_watcher} computes the HTSR layer quality metrics.
It is not fully understood why training via stochastic gradient descent produces heavy tails.

\subsection{The TwoNN intrinsic dimension formula}
\label{subsection_twonn}

We summarise the formulae given by Facco et al. \cite{twonn}. For a detailed derivation, see Appendix \ref{appendix_twonn}.

Suppose we have a collection $X$ of $N$ points $\{ \xsubone , \ldots , \xsubN \}$ in $\RR^d$, that lie on an $m$-dimensional submanifold of $\RR^d$. 
We want to estimate $m$. 
For a given point $\xsubi$, consider the list of its nearest neighbors. 
Let $r_{i,1} \leq  r_{i,2} \leq \ldots$ be a sorted list of their distances from $\xsubi$.
Thus $r_{i,1}$ is the distance from $\xsubi$ to its nearest neighbour, $r_{i,2}$ is the distance to the second-nearest neighbour, and so on.
For each $i$, define $\mu_i = r_{i,2}/r_{i,1}$, the ratio of the distances from $\xsubi$ to its second and first nearest neighbours.
We think of $\{ \mu_i \}_{1 \leq i \leq N}$ as representing draws from the distribution of a random variable $\mu$.
Then it can be shown (see \cite{twonn}, or the Appendix (\ref{appendix_twonn})), that the pdf and CDF of $\mu$ are given by
\begin{equation}
\label{mu_pdf_cdf}
f (t) = m t^{-(m+1)} {\bf{1}}_{[1,\infty)},
\quad
F(x) = (1 - x^{-m}) {\bf{1}}_{[1,\infty)}
\end{equation}

\begin{definition}
The Two Nearest Neighbors (TwoNN) estimator 
for the intrinsic dimension of $X$ is
\begin{equation}
\label{def_twonn_idx}
\TwoNN(X) 
= 
- \EE \left( {\frac{\log(1 - F(\mu))}{\log(\mu)}} \right) 
= 
-{\frac{1}{N}} \sum_{i=1}^N {\frac{\log(1 - F(\mu_i))}{\log(\mu_i)}}
\end{equation}
\end{definition}
The algorithm for estimating this is as follows.
\begin{enumerate}
\item{Compute the pairwise distances for all pairs of points in $X$.}
\item{For each point $\xsubi$, find the two shortest distances $r_{i,1}$ and $r_{i,2}$, and compute $\mu_i = {\frac{r_{i,2}}{r_{i,1}}}$.}
\item{Empirically estimate the cumulative distribution $F(x)$ by sorting the values $\mu_i$ in ascending order.}
\item{Plot the points $(\log(\mu_i), - \log(1 - F(\mu_i)))$ in the plane, and fit a straight line passing through the origin. The estimate for $\TwoNN(X)$ is the gradient of this line.}
\end{enumerate}
In the case where we have a single observation $\mu$, then we can estimate (\ref{def_twonn_idx}) as
\begin{equation}
\label{def_twonn_single_obs}
\TwoNN(X) 
=  - {\frac{\log(0.5)}{\log(\mu)}} = {\frac{\log(2)}{\log(\mu)}}
\end{equation}

\section{The Patnaik-Pearson intrinsic dimension}

We present a simple generative model for a data manifold, and define a new estimator for intrinsic dimension specifically for this situation, which we call the Patnaik-Pearson intrinsic dimension, denoted $\PatnaikPearson(X)$. This arises naturally from applying the TwoNN estimator to this model and then using Patnaik and Pearson's moment-matching formulae. 
In practice we have found numerically, for synthetic data of known intrinsic dimension, that the TwoNN estimator underestimates the real intrinsic dimension (this seems to be well-known). 
We prove various properties of $\PatnaikPearson(X)$, in particular its behaviour when the ESD of $X$ follows a Pareto distribution.

\subsection{A simple Gaussian point cloud generative model for a data manifold}
\label{generative_model}

Consider  a simple generative model for our data manifold $X$, realised as the collection $X = \{ \xsubone, \ldots , \xsubN \}$ of $N$ points in $\RR^d$ defined as follows. 
Let $\xsubi = (x_{i,k})_{1 \leq k \leq d}$, where $x_{i,k} \sim \lambda_k Z_{i,k}$, where the $Z_{i,k}$ are iid $N(0,1)$, and the scalars $\lambda_1$, .. ,$\lambda_d$ satisfy
$0 \leq \lambda_1 \leq \lambda_2 \leq \ldots \leq \lambda_d$.
Denote $\blambda = (\lambda_k)_{1 \leq k \leq d}$.
Obviously, this Gaussian point cloud model is a strong assumption about the shape of our data manifold and there are many data manifolds that do not correspond to this, but it is a useful model for our analysis. 
In general, given some estimator $\ID(X)$ for the intrinsic dimension of $X$, how does this depend on $\blambda$?
Can we find functions $g_1$, $g_2$ that are lower and upper bounds (for sufficiently large $N$ and $d$): 
\begin{equation}
\label{bounds_one_g}
g_1 (\blambda ) \leq \ID(X) \leq g_2 (\blambda)
\end{equation}
In particular, for the situation where the $\lambda_k$ are drawn from a Pareto distribution with tail density exponent $\alpha$ (\ref{pareto_alpha_pdf_X}), we would like to characterise $\ID(X)$ in terms of $\alpha$, i.e. find functions $h_1$, $h_2$ such that
\begin{equation}
\label{bounds_two_h}
h_1 (\alpha) \leq \ID(X) \leq h_2 (\alpha)
\end{equation}

\begin{definition}
\label{sensible_estimator}
We expect any sensible estimator $\ID(X)$ to have the following properties:
\begin{enumerate}
\item{$1 \leq \ID(X) \leq d$, provided $\blambda \neq {\bf 0}$.}
\item{For any $m$ with $0 \leq m \leq d - 1$, if $\lambda_1 = \ldots = \lambda_{d - m} = 0 \neq \lambda_{d - m + 1}$, then $\ID(X) \leq m$.}
\item{For any $m$ with $0 \leq m \leq d - 1$, if $\lambda_1 = \ldots = \lambda_{d - m} = 0$, and $\lambda_{d - m + 1} = \ldots = \lambda_d \neq 0$, then $\ID(X) = m$.}
\item{If we scale each $\lambda_k$ by a strictly positive scalar $\phi$, then $\ID(X)$ will be unchanged.}
\item{$\ID(X)$ is invariant under rotations and translations of $X$.}
\end{enumerate}
\end{definition}

\subsection{The generalised Chi-Squared distribution}

We now apply the TwoNN estimator to our generative model and find that the Patnaik-Pearson estimator naturally emerges from our analysis. 
To find nearest neighbours we need to calculate the distances between points.
The (squared) distance between two randomly chosen points $\xsubi$, $\xsubj \in X$ is given by
$$
{\dist(\xsubi , \xsubj )}^2 = {|| \xsubi - \xsubj ||}^2 = \sum_{k=1}^d \lambda_k^2 (Z_{i,k} - Z_{j,k})^2
$$
Now, $Z_{i,k} - Z_{j,k} = \sqrt{2} Z_{i,j,k}$, for some $Z_{i,j,k} \sim N(0,1)$.
Hence 
\begin{equation}
\label{dist_sqrd}
{\dist(\xsubi , \xsubj )}^2 = 2 \sum_{k=1}^d \lambda_k^2 Z_{i,j,k}^2
\end{equation}
which is ``Chi-squared with $2 (\sum_{k=1}^d \lambda_k^2)$ degrees of freedom".
For each point $\xsubi$ we want to find the distances to its first and second neighbours, namely
$$
r_{i,1} = 
\min_{j \neq i} { \{ \dist(\xsubi , \xsubj ) \} },
\quad
r_{i,2} =  
\secondmin_{j \neq i} { \{ \dist(\xsubi , \xsubj ) \} }
$$

{\bf Patnaik-Pearson moment-matching}: Suppose that $Z_i \sim N(0,1)$ are iid, for $1 \leq i \leq d$. 
As before, assume that 
$0 \leq \lambda_1 \leq \lambda_2 \leq \ldots \leq \lambda_d$,
and define the weighted sum of squares
\begin{equation}
\label{generalised_chi_squared}
Y = \lambda_1 Z_1^2 + \lambda_2 Z_2^2 + \ldots + \lambda_d Z_d^2
\end{equation}
This follows a Generalised Chi-Squared distribution.
Patnaik \cite{patnaik} showed that this could be approximated (matching the first two moments of the distribution) by a scaled $\chi^2$-distribution
\begin{equation}
\label{def_nu_def_c}
Y \sim c \cdot \chi^2 (\nu) \sim \Gamma({\tfrac{\nu}{2}}, 2c)
\quad \mathrm{where} \quad
\nu = \nu({\blambda}) = {\frac 
{(\sum_{i=1}^d {\lambda_i})^2}
{\sum_{i=1}^d \lambda_i^2}
}
, \quad
c = c( {\blambda}) = {\frac {\sum_{i=1}^d {\lambda_i^2} }{\sum_{i=1}^d {\lambda_i}} 
}
\end{equation}
This was extended by (Egon) Pearson \cite{pearson} to match the first three moments of the distribution. 
It is immediate that $1 \leq \nu (\blambda) \leq d$,
with the upper and lower bounds being realised by the cases $\lambda_1 = \ldots = \lambda_d \neq 0$; and $\lambda_i = 0$  for all $i \neq d$, respectively. 
Now, $\Gamma({\tfrac{\nu}{2}}, 2c)$ has pdf given by 
\begin{equation}
\label{pdf_gamma}
f(t) = {\frac{1}{\Gamma({\tfrac{\nu}{2}})}} 
	2^{\nu/2} 
	t^{\nu/2 -1} 
	e^{-t/2}, 
	\quad t > 0
\end{equation}
and hence, for $x$ small, CDF given by 
$
F(x) \approx {\frac{1}{\Gamma({\tfrac{\nu}{2}} +1)}} 
{\left( 
{\frac{x}{2}} 
\right)
}^{\nu/2} 
$.
Starting with $Y \sim c \cdot \chi^2 (\nu)$ 
, then $\chi^2 (\nu)$ has pdf proportional to
$t^{\nu/2 -1} e^{-t/2}$
so near $x=0$, 
${\CDF}_{\chi^2 (\nu)} (x) \sim {\frac{1}{\Gamma({\frac{\nu}{2}} + 1)}} \left({\frac{x}{2}}\right)^{\nu/2}$.
So to solve 
$\CDF_{\chi^2 (\nu)} (x) = y$, 
then 
$\left({\frac{x}{2}}\right)^{\nu/2} = y \Gamma ( {\tfrac{\nu}{2}} + 1)$,
which implies 
$x = 2 y^{2/\nu} ( \Gamma ( {\tfrac{\nu}{2}} + 1) )^{2/\nu}$.

Suppose we have a non-negative random variable $X$.
If we draw $N$ times from this distribution, obtaining $X_1$, $X_2$, .. , $X_N$, and define the first and second order statistics
$$X_{(1)} = \min(X_1, X_2, \ldots, X_N), \quad  X_{(2)} = \secondmin(X_1, X_2, \ldots, X_N)$$
Then 
$\EE(X_{(1)}) = {\CDF}^{-1}({\tfrac{1}{N}})$ and
$\EE(X_{(2)}) = {\CDF}^{-1}({\tfrac{2}{N}})$.
Hence, for $Y \sim c \cdot \chi^2 (\nu)$, then
$\PP( Y \leq x) = \PP( \chi^2(\nu) \leq {\tfrac{x}{c}})$,
and ${\CDF}_Y (x) = {\CDF}_{\chi^2 (\nu)} ( {\frac{x}{c}})$.
It follows that 
${\CDF}_Y^{-1} ({\tfrac{1}{N}}) 
= c \, {\CDF}_{\chi^2 (\nu)}^{-1} ({\tfrac{1}{N}})$.
Therefore
$$ {\CDF}_Y^{-1} ({\tfrac{2}{N}}) 
= c \, {\CDF}_{\chi^2 (\nu)}^{-1} ({\tfrac{2}{N}}) 
\approx c \, 2^{2/\nu} \, {\CDF}_{\chi^2 (\nu)}^{-1} ({\tfrac{1}{N}})
= 2^{2/\nu} \, {\CDF}_Y^{-1} ({\tfrac{1}{N}})$$
hence
${\CDF}_Y^{-1} ({\tfrac{2}{N}}) / {\CDF}_Y^{-1} ({\tfrac{1}{N}})  
\approx 2^{2/\nu}$
and thus, using (\ref{dist_sqrd}), 
$$\mu = \mu(N) := \sqrt{{\CDF}_Y^{-1} ({\tfrac{2}{N}}) / {\CDF}_Y^{-1} ({\tfrac{1}{N}})}   
\approx 2^{1/\nu}$$
Combining this with (\ref{def_twonn_single_obs}) gives 
\begin{equation}
\label{twonn_meets_patnaik_pearson}
\TwoNN(X) \approx {\tfrac{\log(2)}{\log(\mu(N))}} 
= {\tfrac{\log(2)}{\log(2^{1/\nu})}}  
= \nu (\blambda)
\end{equation}
This suggests that $\nu (\blambda)$ may be a useful proxy for $\TwoNN(X)$. 

\subsection{The Patnaik-Pearson intrinsic dimension of a data manifold}

Following (\ref{def_nu_def_c}, \ref{twonn_meets_patnaik_pearson}), we define a function 
$\nu : \RR_{\geq 0}^d \backslash \{ {\bf 0} \} \rto \RR_{> 0}$,
and use this to define the Patnaik-Pearson intrinsic dimension $\PatnaikPearson(X)$ of a (realisation of a) data manifold $X$.

\begin{definition}
\label{defn_nu} 
For $\blambda \in \RR_{\geq 0}^d \backslash \{ {\bf 0} \}$, $\blambda = {( \lambda_i )}_{1 \leq i \leq d}$, define 
$\nu( \blambda) = 
{\frac 
{(\sum_{i=1}^d {\lambda_i})^2}
{\sum_{i=1}^d \lambda_i^2}
}$. 
\end{definition}
Obviously $\nu$ extends to a function $\RR^d \backslash \{ {\bf 0} \} \rto \RR_{>0}$, 
but we want to emphasise the fact that we are assuming that all the $\lambda_i$ are non-negative.
Note also that this is the same $\nu(\blambda)$ as in (\ref{def_nu_def_c}).

\begin{lemma}
\label{lemma_nu_lambda_basic_properties}
$\nu(\blambda)$ satisfies properties 1 to 4 of Definition \ref{sensible_estimator}.
\end{lemma}
{\bf Proof:} These are all straightforward. $\myqed$

\begin{definition} The Patnaik-Pearson intrinsic dimension $\PatnaikPearson(X)$ of a data manifold $X$. 
Given a realisation of $X$ as $N$ points $\bfxone, \bfxtwo, \ldots, \bfxN \in \RR^d$, considered as an $N \times d$ matrix whose rows are the $\bfxi^T$,
define $\xoverline = {\tfrac{1}{N}} \sum_{i=1}^N \bfxi = {\tfrac{1}{N}} X^t {\bf 1}_N$. 
We centre the data at ${\bf 0}$ by performing column-demeaning,
\begin{equation}
\label{def_coldem} 
\Xcoldem = X - {\bf 1}_N {\xoverline}^T = X - {\tfrac{1}{N}} ({\bf 1}_N {\bf 1}_N^T) X = X - {\tfrac{1}{N}} JX
\end{equation}
i.e. we subtract $\xoverline$(as a row vector) from each individual row of $X$.
Using (\ref{svd}) we have
$\Xcoldem = U S V^T$
where $S$ is a non-negative diagonal $N \times d$ matrix, with diagonal elements $\lambda_i$ satisfying $0 \le \lambda_1 \le \ldots \le \lambda_d$, 
and $U$ and $V$ are 
orthogonal matrices, of dimension $N \times N$ and $d \times d$. 
We define the Patnaik-Pearson intrinsic dimension of $X$ as
\begin{equation}
\label{pp_int_dim}
\PatnaikPearson(X) = \nu(\blambda) = 
{\frac 
{(\sum_{i=1}^d {\lambda_i})^2}
{\sum_{i=1}^d \lambda_i^2}
}
= {\frac{\Tr(S)^2}{\Tr(S^t S)}}
\end{equation}
\end{definition}
Note that this resembles, but is distinct from, the stable rank of a matrix, which for $\Xcoldem$ is
\begin{equation}
\label{stable_rank}
\stablerank(\Xcoldem) = 
{\normXresid}_F^2 
/ {\normXresid}_{op}^2 
= \sum_{i=1}^d 
\lambda_i^2 / 
{\lambda_{\max}^2}
\end{equation}
The function $\nu(\blambda)$ appears in a wide range of situations under various names, for example ``participation ratio", ``concentration ratio", ``effective dimension" or ``effective sample size".
Most relevant for this work, the inverse of $\nu(\blambda)$ appears in Cizeau and Bouchaud \cite{levy_matrices} as the ``inverse participation ratio". 

Note also that we have defined $\PatnaikPearson(X)$ for an $N \times d$ data manifold $X$, but the definition makes sense for any matrix, so given (dimensionally compatible) matrices $A$, $B$ we can consider $\PatnaikPearson(A)$, $\PatnaikPearson(B)$ and $\PatnaikPearson(AB)$.  
In particular we can apply this to the weight matrices $W$ of neural networks, and compare $\PatnaikPearson(XW)$ with $\PatnaikPearson(X)$ and $\PatnaikPearson(W)$.

\begin{definition} 
In the sequel we will also work extensively with the normalised Patnaik-Pearson dimension ${\tfrac{1}{d}} \PatnaikPearson(X)$, which we will sometimes refer to as ``nu/d". 
\end{definition}

\begin{theorem} The function 
${\RR}^{N \times d} \backslash \{ {\bf 0} \} \rightarrow {\RR}_{>0}$, $X \mapsto \PatnaikPearson(X)$
is continuous, for the topology on $N \times d$ matrices induced by the Frobenius norm (and the usual topology on $\RR$).
\end{theorem}
{\bf Proof:} The Frobenius (Hilbert-Schmidt) norm of a matrix $A$ is given by ${||A||}_F := ( \sum_{ij} |A_{ij}|^2 )^{1/2}$. 
It is immediate that the map $X \mapsto \Xcoldem$ is continuous for the Frobenius norm. 
Continuity of $\Xcoldem \mapsto S$, where $\Xcoldem = U S V^t$ is the (unique) SVD of $\Xcoldem$ with $S = \diag(\blambda)$, with $0 \le \lambda_1 \le \ldots \le \lambda_d$ follows from Mirsky's theorem \cite{bhatia_matrix_analysis}.
This states, if $A$, $B$ are matrices with singular values $\{ \sigma_i(A) \}$, $\{ \sigma_i (B) \}$, then 
$$(\sum_i ( \sigma_i(A) - \sigma_i(B) )^2 )^{1/2} \le {|| A - B ||}_F$$
Finally, continuity of $\nu : \RR^d \backslash \{ {\bf 0} \} \rightarrow \RR_{>0}$ follows from Definition (\ref{defn_nu}).
Hence $X \mapsto \PatnaikPearson(X)$ is the composition of three continuous functions and is therefore continuous.
$\myqed$

\begin{theorem}
\label{thm_pp_invariant_under_orthogonals_and_transpose} 
Provided $X = \Xcoldem$, if $O$ is an orthogonal $N \times N$ matrix, and $Q$ an orthogonal $d \times d$ matrix, then 
\begin{equation}
\label{pp_invariant_under_orthogonals}
\PatnaikPearson(OX) = \PatnaikPearson(X) = \PatnaikPearson(XQ) = \PatnaikPearson(OXQ)
\end{equation}
and further
\begin{equation}
\label{pp_invariant_under_transpose}
\PatnaikPearson(X^T) = \PatnaikPearson(X)
\end{equation}
\end{theorem}
{\bf Proof:} 
These are all immediate from the definition.
$\myqed$\\

Together, Lemma \ref{lemma_nu_lambda_basic_properties} and Theorem \ref{thm_pp_invariant_under_orthogonals_and_transpose} show that $\PatnaikPearson(X)$ satisfies properties 1 to 5 of Definition \ref{sensible_estimator}.
We establish a further useful result:

\begin{figure}[htbp]
    \centering
    \includegraphics[width=0.5\textwidth]{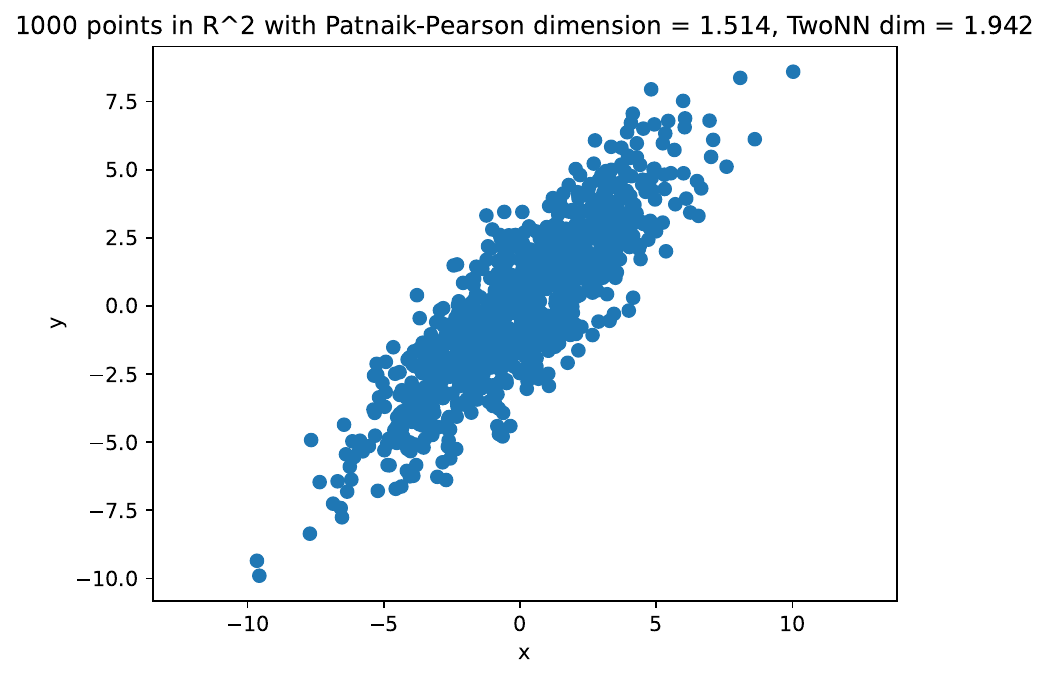}
    \caption{1000 points in $\RR^2$ with Patnaik-Pearson dimension 1.514, TwoNN dimension 1.942. This suggests that the Patnaik-Pearson dimension may be thought of as a ``global" measure of dimension, whereas the TwoNN dimension captures local dimensionality.}
    \label{fig:R_2_nu_eq_1_489_N_eq_1000.pdf}
\end{figure}

\begin{figure}
    \centering
	\includegraphics[width=0.45\textwidth]{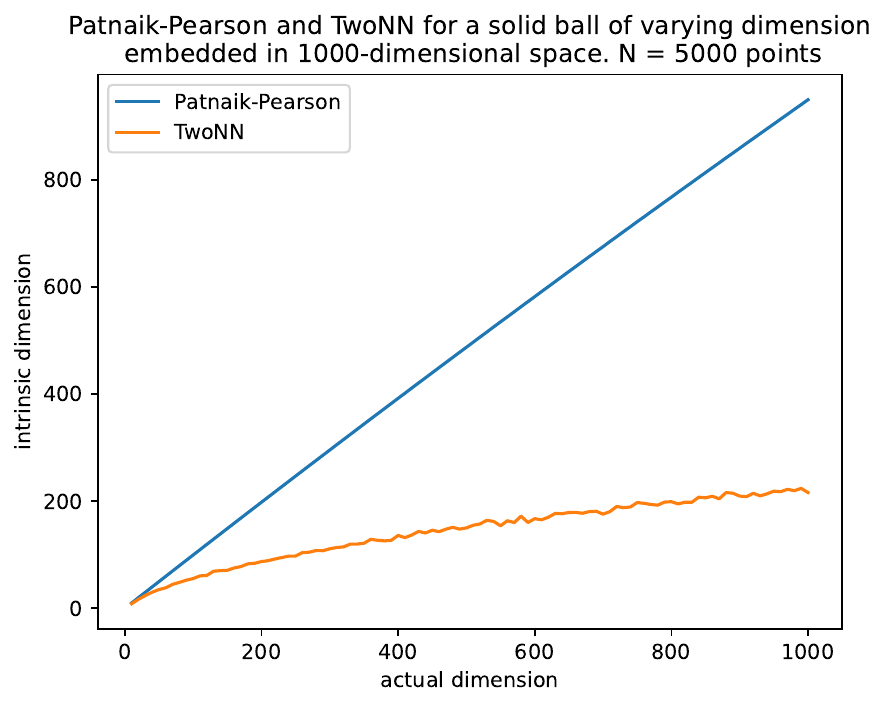}	
	\includegraphics[width=0.45\textwidth]{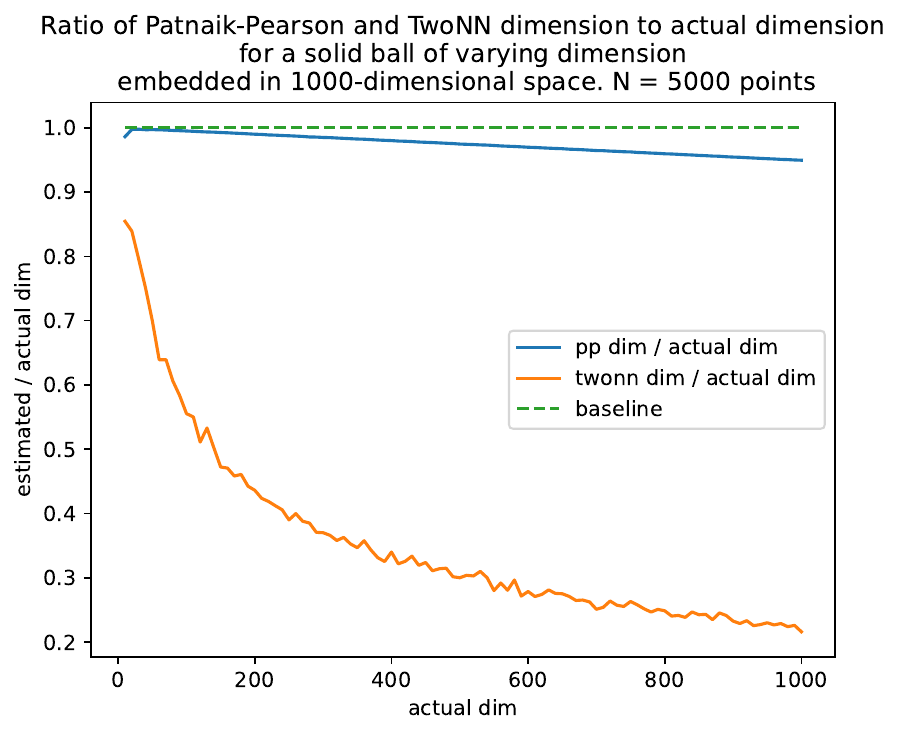}
    \caption{Patnaik-Pearson and TwoNN dimension estimates for a solid ball of dimension varying between 10 and 1000, in $\RR^{1000}$.} 
    \label{fig:pp_twonn_ball_d1000_two_figs.pdf}
\end{figure}

\begin{theorem}
\label{thm_h_Y_d} 
Suppose $Y$ is a non-negative random variable with finite moments $e_k = \EE(Y^k) < \infty$ for $k = 1, 2$, and $\EE(Y^2) \neq 0$.
Given iid $Y_i \sim Y$, for $i = 1, 2, \ldots d$, define
$$
h(Y,d) 
= {\frac{1}{d}} {\frac{(\sum_{i=1}^d Y_i)^2}{\sum_{i=1}^d Y_i^2}} 
= {\frac{ ({\tfrac{1}{d}} \sum_{i=1}^d Y_i)^2}{ {\tfrac{1}{d}} \sum_{i=1}^d Y_i^2}} 
$$
Then 
\begin{equation}
\label{lim_d_infty_h_Y_d}
\lim_{d \rto \infty} h(Y,d) = {\frac{ \EE(Y)^2}{\EE(Y^2)}}
\end{equation}
Furthermore, for $e_3$ and $e_4$ also finite, then 
$h(Y,d)
\approx
\left(
{\frac{ \EE(Y)^2}{\EE(Y^2)}}
\right) 
\left[
1 + {\tfrac{1}{\sqrt{d}}} \sigma(Q) Z
\right]
$, 
where $Z \sim N(0,1)$ and $\sigma(Q)$ is a constant that we find in (\ref{hY_final}).
\end{theorem}
{\bf Proof:} Since $e_1$ and $e_2$ are finite, then by Kolmogorov's Strong Law of Large Numbers (SLLN) \cite{modelling_extremal}
${\tfrac{1}{d}} \sum_{i=1}^d Y_i$ and ${\tfrac{1}{d}} \sum_{i=1}^d Y_i^2$ converge almost surely to $\EE(Y)$ and $\EE(Y^2)$ respectively,
and (\ref{lim_d_infty_h_Y_d}) follows. 
To estimate the rate of convergence, for large $d$,
$$
{\tfrac{1}{d}} \sum_{i=1}^d Y_i \sim e_1 + {\tfrac{\sigma(Y)}{\sqrt{d}}} Z_1
\quad \implies \quad
{\tfrac{1}{d^2}} (\sum_{i=1}^d Y_i)^2 \sim e_1^2 + {\tfrac{2 e_1 \sigma(Y)}{\sqrt{d}}} Z_1 + O( {\tfrac{1}{d}})
$$
and also, provided $e_3$ and $e_4$ are finite,
${\tfrac{1}{d}} \sum_{i=1}^d Y_i^2 \sim e_2 + {\tfrac{\sigma(Y^2)}{\sqrt{d}}} Z_2$
where $Z_1$, $Z_2$ are $N(0,1)$. 
Thus
\begin{equation}
\label{hY_intermediate}
h(Y,d) \approx \formulaFive
\end{equation}
Now, $Z_1$ and $Z_2$ are correlated, with
$$
\rho = \Corr(Z_1, Z_2)  
= \Cov 
\left(
{\tfrac{1}{\sigma(Y) \sqrt{d}}} \sum_{i=1}^d Y_i,
{\tfrac{1}{\sigma(Y^2) \sqrt{d}}} \sum_{j=1}^d Y_j^2
\right)
= 
{\frac{1}{d}}
\left(
{\frac{e_3 - e_2 e_1}{\sigma(Y) \sigma(Y^2)}}
\right)
= 
{\frac{1}{d}} \Corr( Y^2, Y)
$$
In general, if $Q = a Z_1 + b Z_2$, with $\Corr(Z_1,Z_2) = \rho$, then $Q = \sigma(Q) Z$, with $Z \sim N(0,1)$ and
$\sigma(Q)^2 = a^2 + 2 \rho ab + b^2$. 
Applying this to (\ref{hY_intermediate}) gives
\begin{equation}
\label{hY_final}
h(Y,d)
\approx
\left(
\fracEOneSqrdOverETwo
\right) 
\left[
1 + {\tfrac{1}{\sqrt{d}}} \sigma(Q) Z
\right]
\quad with \quad
\sigma(Q) = 
{\left(  
{\frac{4 {\sigma(Y)}^2}{e_1^2}} 
- 
{\frac{4(e_3 - e_2 e_1)}{e_2 e_1}}
+ {\frac{{\sigma(Y^2)}^2}{e_2^2}}
\right)}^{1/2}
\end{equation}
This completes the proof. $\myqed$\\

Our definition of Patnaik-Pearson intrinsic dimension was motivated by the Gaussian point cloud generative model presented in Section \ref{generative_model}.
It is natural to ask, for what class of data distributions beyond this point cloud model does $\PatnaikPearson(X)$ give a consistent estimator of the true intrinsic dimension of the data manifold $X$?
Noting the universality results of the L2N2 estimator of Ong et al. \cite{universal_nne},
are there any corresponding results for $\PatnaikPearson(X)$, or does it have a systematic
bias for specific distribution families?

\subsection{Uniform distribution}
Suppose the $\lambda_i$ are drawn from a uniform $U[0,1]$ distribution.
For $Y \sim U[0,1]$, then $\EE(Y^k) = {\tfrac{1}{k+1}}$. 
Hence by (\ref{hY_final}),
\begin{equation}
\label{nu_over_d_uniform}
h(Y,d) 
\approx {\frac{3}{4}} 
\left[ 
1 + 
{\frac{\sigma(Q)}{\sqrt{d}}}
Z 
\right] 
\quad 
{\mathrm{as}} 
\quad d \rto \infty
\end{equation} 
with $\sigma(Q)  = \sqrt{{\tfrac{2}{15}}}$. 
So for large $d$, ${\frac{1}{d}} \nu(\blambda) \approx {\frac{3}{4}}$, with standard deviation 
${\frac{3 \sigma(Q)}{4 \sqrt{d}}} = {\frac{0.2739}{\sqrt{d}}}$.
This is confirmed numerically by Figure \ref{fig:nu_over_d_uniform_N_10000_d_100.pdf}(b), with predicted and observed standard deviations of 0.0274 and 0.0272.

\begin{figure}
    \centering
	\includegraphics[width=0.45\textwidth]{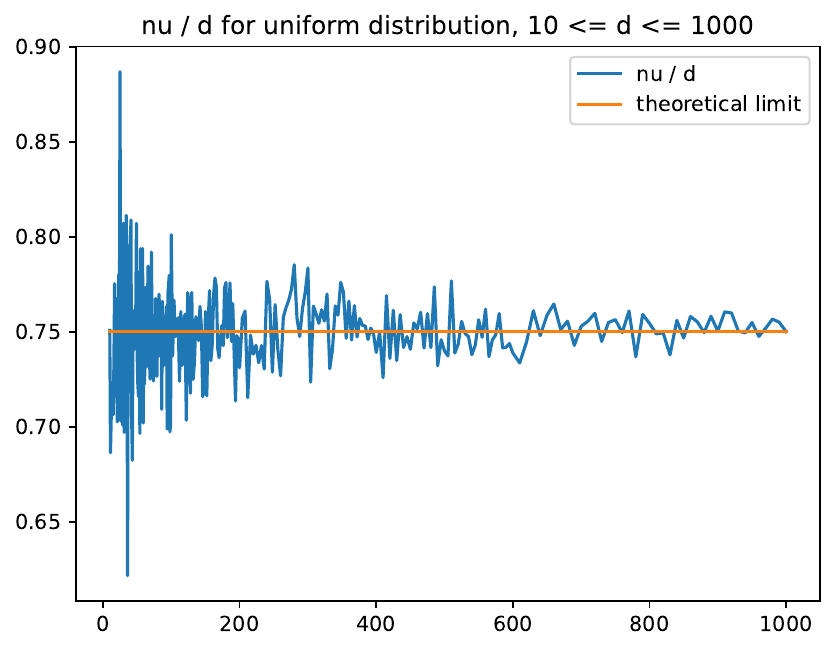} 
    \includegraphics[width=0.45\textwidth]{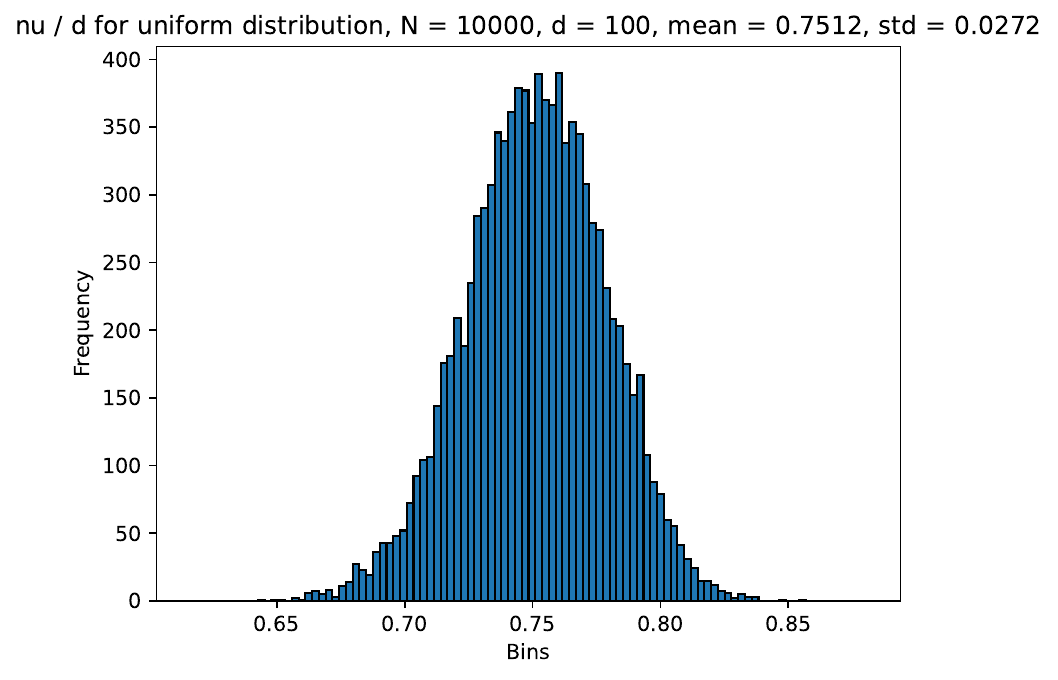}
    \caption{Uniform distribution : numerical tests of (\ref{nu_over_d_uniform}), (a) for $10 \leq d \leq 1000$, (b) for 10,000 samples with $d = 100$} 
    \label{fig:nu_over_d_uniform_10_leq_d_leq_1000.pdf}
	\label{fig:nu_over_d_uniform_N_10000_d_100.pdf}
\end{figure}

\subsection{Pareto distribution with tail exponent $\alpha$}
\label{subsection_pp_dim_meets_pareto}

Suppose the $\lambda_i$ are drawn from a Pareto distribution with tail exponent $\alpha > 0$ (\ref{pareto_alpha_pdf_X}),
with pdf
\begin{equation}
\label{pareto_alpha_pdf}
f(t) = {\frac{\alpha}{t^{\alpha+1}}}, \, t \geq 1, \quad f(t) = 0, \, t < 1
\end{equation}
Define 
\begin{equation}
\label{defn_s_C_alpha_nu_alpha_d}
s = {\frac{1}{\alpha}}, 
\quad 
C(\alpha) = {\frac{(1-2s)}{(1-s)^2}} = {\frac{\alpha(\alpha - 2)}{(\alpha - 1)^2}}, \,\, \alpha \neq 1,
\quad
\nu(\alpha,d) = \EE \left(
{
\frac
{
(\sum_{i=1}^d \lambda_i)^2
}
{
\sum_{i=1}^d \lambda_i^2
}
}
\right)
\end{equation}

\begin{theorem} Define 
$\nuinfty (\alpha) = \lim_{d \rto \infty} {\tfrac{1}{d}}\nu(\alpha,d)$. 
Then 
\begin{equation}
\label{defn_nu_infty}
\nuinfty (\alpha) =  \left\{
\begin{aligned}
C(\alpha) \, : \, \alpha \geq 2 \\
0 \, : \, \alpha \leq 2
\end{aligned}
\right.
\end{equation}
This is illustrated in Figure \ref{fig:nu_over_d_as_function_of_alpha.pdf}.
\end{theorem}
{\bf Proof:} For $\alpha > 2$, by Theorem \ref{thm_h_Y_d}, $\nu_\infty (\alpha)$ = $\EE(\lambda)^2 / \EE(\lambda^2)$, with
$\EE(\lambda) = {\frac{\alpha}{\alpha - 1}}$ 
and 
$\EE(\lambda^2) = {\frac{\alpha}{\alpha - 2}}$.

For $2 > \alpha > 1$, then $\EE(Y)$ is finite, but $\EE(Y^2)$ is infinite. By Kolmogorov's SLLN \cite{modelling_extremal},
${\tfrac{1}{d}} \sum_{i=1}^d \lambda_i$ converges almost surely to $\EE(\lambda) < \infty$, but ${\tfrac{1}{d}} \sum_{i=1}^d \lambda_i^2$ diverges.
Hence ${\tfrac{1}{d}} \nu(\alpha,d) \rto 0$ as $d \rto \infty$.

For $1 > \alpha$, then both 
${\tfrac{1}{d}} \sum_{i=1}^d \lambda_i$ 
and 
${\tfrac{1}{d}} \sum_{i=1}^d \lambda_i^2$ 
diverge. 
In this extremely heavy tailed situation, the largest value 
$\lambda_{\max}$ 
dominates (\ref{single_big_jump}), with
$\nu(\alpha,d) 
\approx 
{\frac{({\lambda_{\max}})^2}{\lambda_{\max}^2}} 
= 1$ 
for large $d$.
So 
${\tfrac{1}{d}} \nu(\alpha,d) \rto 0$.
$\myqed$\\

\begin{figure}
    \centering
    \includegraphics[width=0.45\textwidth]{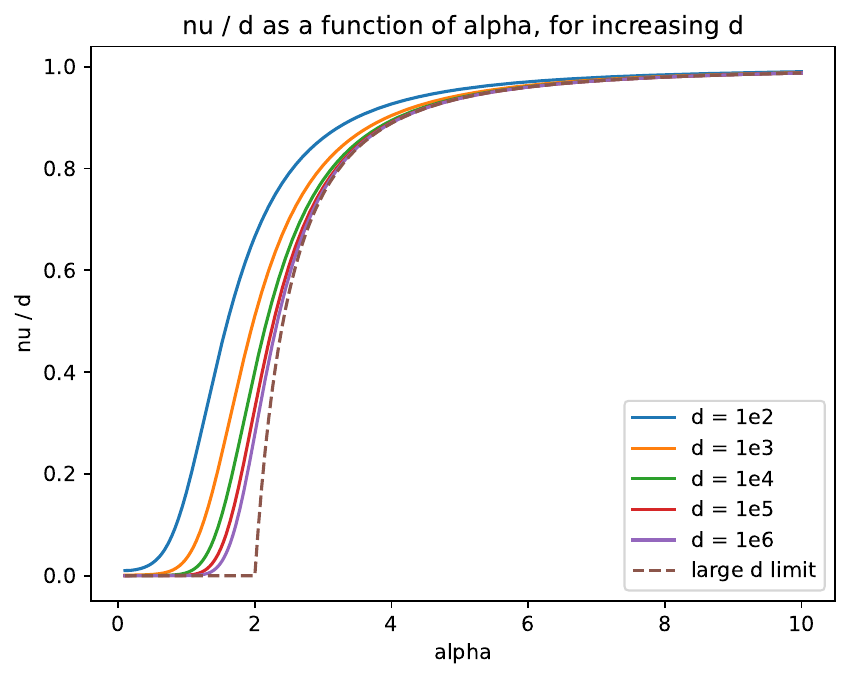}
    \includegraphics[width=0.45\textwidth]{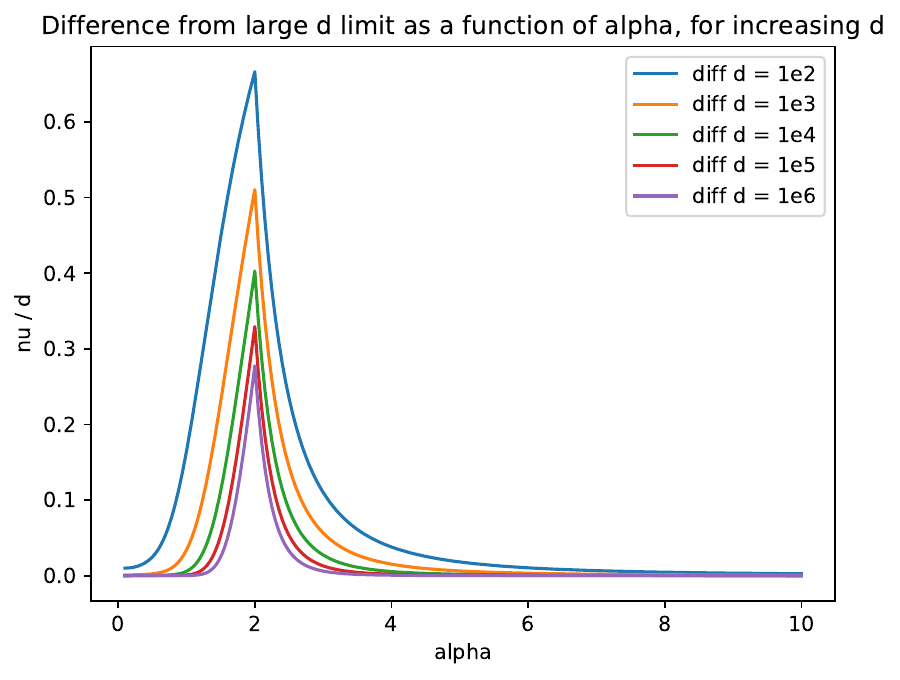} 
    \caption{(a) ${\tfrac{1}{d}}\nu(\alpha,d)$ as $\alpha$ varies, for $d$ increasing from 100 to 1,000,000, together with $\nuinfty(\alpha)$.\\ (b) Difference of ${\tfrac{1}{d}}\nu(\alpha,d)$ from conjectured limit as $\alpha$ varies, for $d$ increasing from 100 to 1,000,000.} 
	\label{fig:nu_over_d_as_function_of_alpha.pdf}
	\label{fig:nu_over_d_difference_from_large_d_limit.pdf}
\end{figure}

We now address the rate of convergence of ${\tfrac{1}{d}} \nu(\alpha,d)$ to $\nu_\infty (\alpha)$, for $\alpha \geq 1$.
The CDF for (\ref{pareto_alpha_pdf}) is:
\begin{equation}
\label{pareto_alpha_CDF}
F(x) = \PP( \lambda \leq x) = \int_1^x  {\frac{\alpha}{t^{\alpha+1}}} \, dt  = 1 - x^{-\alpha}
\end{equation}
Hence $F^{-1} (y) = (1 - y)^{-1/\alpha}$, for $0 \leq y \leq 1$, and
$\lambda_k = F^{-1} \left(  {\frac{k + 1}{d + 2}}   \right) = {\left( {\frac{d + 1 - k}{d + 2}} \right)}^{-1/\alpha}$.
Using (\ref{defn_s_C_alpha_nu_alpha_d}), 
\begin{equation}
\label{sum_lambda_k_squared_integral}
\sum_{k=1}^d \lambda_k 
\approx d^s \int_1^d {\frac{1}{x^s}} \, dx, \quad \quad
\sum_{k=1}^d \lambda_k^2 
\approx d^{2s} \int_1^d {\frac{1}{x^{2s}}} \, dx
\end{equation}
So for $\alpha \neq 1, 2$, then 
$\sum_{k=1}^d \lambda_k \approx {\frac{d^s}{1-s}} \left[ d^{1-s} - 1 \right]$,
$\sum_{k=1}^d \lambda_k^2 \approx {\frac{d^{2s}}{1-2s}} \left[ d^{1-2s} - 1 \right]$
Hence
\begin{equation}
\label{nu_alpha_d_one}
\nu(\alpha, d) \approx {\frac{(1-2s)}{(1-s)^2}} \cdot {\frac{[ d^{1-s} - 1]^2}{[ d^{1-2s} - 1]}} 
= C( \alpha) {\frac{[ d^{1-s} - 1]^2}{[ d^{1-2s} - 1]}}
\end{equation}
We treat the four cases $\alpha > 2$, $\alpha = 2$, $2 > \alpha > 1$ and $\alpha = 1$ separately.
For reasons of space, we omit the detailed working, the results are shown in Table \ref{table:nu_alpha_d_summary}.

\begin{table}[h!]
\centering
\begin{tabular}{|c|c|c|}
\hline
Tail exponent $\alpha$ & $\nu(\alpha,d)$ for large $d$ & ${\tfrac{1}{d}} \nu(\alpha,d)$ for large $d$ \\ \hline
$\alpha > 2$ & $C(\alpha) d$ & $C(\alpha)$ \\ \hline
$\alpha = 2$ & ${\tfrac{4d}{\ln(d)}}$ & ${\tfrac{4}{\ln(d)}}$ \\ \hline
$2 > \alpha > 1$ & $-C(\alpha) d^{2 \left( {\tfrac{\alpha - 1}{\alpha}} \right)}$ & $-C(\alpha) d^{- \left( {\tfrac{2 - \alpha}{\alpha}} \right)}$ \\ \hline
$\alpha = 1$ & $(\ln(d))^2$ & ${\tfrac{1}{d}}(\ln(d))^2$ \\ \hline
\end{tabular}
\caption{Summary of the behaviour of $\nu(\alpha,d)$ for large $d$.}
\label{table:nu_alpha_d_summary}
\end{table}

In practice, for a weight matrix $W$ in the HTSR ``Heavy (Fat) Tailed" phase 
the ESD is a mixture of bulk (described by Marchenko-Pastur) and tail (described by Pareto).
Quantifying how the Patnaik-Pearson dimension is influenced by the boundary between the two would refine our understanding.

\section{Application to Neural Networks}

We now study the behaviour of the Patnaik-Pearson dimension of a data manifold under the transformations commonly seen in neural networks. 
We start by establishing the correspondence between the behaviour of the Patnaik-Pearson dimension for Pareto distributions described in Section \ref{subsection_pp_dim_meets_pareto}, and the HTSR and SETOL theory outlined in Section \ref{subsection_htsr_setol}.
In particular, once we take account of the different notational convention (\ref{alpha_HTSR_alpha_Pareto}), the critical values for the two theories correspond exactly with one another. 

We then investigate the behaviour of the Patnaik-Pearson dimension under various transformations of the data manifold $X$ - matrix multiplication by a weight matrix $W$; application of softmax; application of activation functions such as ReLU; addition (as seen in skip or residual connections) interpolation and concatenation of two data manifolds; normalisation; and application of the attention kernel.

A natural further development of this work would be; rather than studying individual layer operations separately, 
derive a differential equation governing the evolution of $\PatnaikPearson(X)$ as $X$ passes through many layers, 
in the spirit of the mean-field interacting particle system approach of Rigollet et al. \cite{math_pers_trans, mean_field_dyn}.

\subsection{The correspondence between HTSR and SETOL and the Patnaik-Pearson dimension}
\label{correspondence_htsr_setol_pp_dim}

Suppose we have an $d \times d$ weight matrix $W$ with SVD $W = U S V^T$, using (\ref{svd}), 
where $U$ and $V$ are $d \times d$ orthogonal matrices, and $S$ is a diagonal matrix with diagonal entries $0 \leq \lambda_1 \leq \lambda_2 \leq \ldots \leq \lambda_d$. 
Further suppose the $\lambda_i$ have been drawn from a Pareto distribution with tail exponent $\alpha(W)$, i.e. 
\begin{equation}
\label{pareto_alphaW}
\PP( \lambda \geq t ) = t^{-\alpha(W)}
\end{equation}
Suppose that ${\overline{\bf{w}}} := {\tfrac{1}{d}} \sum_{i=1}^d {\bf w}_i = {\bf 0}$, 
where the ${\bf w}_i$ are the rows of $W$. 
In this case $\PatnaikPearson(W)$ (considered as a data manifold of $d$ points in $\RR^d$) is given by $\nu(\blambda)$.
Then
$$Y = W^T W = (U S V^T )^T (U S V^T) = V S^T S V^T = V \Lambda V^T$$
where $\Lambda$ is diagonal with entries $\lambda_i^2$.
Using (\ref{pareto_alphaW}), 
$\PP( \lambda^2 \geq t ) 
= \PP( \lambda \geq \sqrt{t} ) 
= t^{-\alpha(W)/2}$.
In HTSR notation, by (\ref{htsr_alpha}), $\PP( \lambda^2 \geq t ) \sim t^{1 - \alpha_{\HTSR}(W)}$.
So
$\alpha_{\HTSR} (W)$, the tail exponent of the ESD of $W^T W$, satisfies
\begin{equation}
\label{alpha_htsr_pareto}
1 -  \alpha_{\HTSR} (W) = - \alpha(W)/2 \quad \implies \quad 
\alpha(W) = 2(\alpha_{\HTSR} (W) - 1)
\end{equation}
In particular,
$2 \leq \alpha_{\HTSR} (W) \leq 6$ implies $2 \leq \alpha (W) \leq 10$,
and the critical value $\alpha_{\HTSR} (W) = 2$ 
discovered in the HTSR work corresponds to $\alpha (W) = 2$, which is the critical value for (\ref{defn_nu_infty}).
By consulting Figure \ref{fig:nu_over_d_as_function_of_alpha.pdf}, we see that for $d \approx 1000$, $\alpha = 2$ corresponds to ${\tfrac{1}{d}} \PatnaikPearson(X) \approx 0.6$, which is useful to keep in mind for the numerical results in the sequel.

\subsection{Transformer Architecture}

A generic transformer can be visualised as a stack of encoder and/or decoder layers. Each layer can be broken down into simpler components which we analyse individually. 
Following \cite{complete_mot}, a single encoder layer with normalisation can be realised as:
\begin{equation}
\label{encoder_layer} 
Z(1) = X + \MultiHead(\LayerNorm(X)), \quad
Z(2) = Z(1) + \FFN(\LayerNorm(Z(1)))
\end{equation}
Here $X$, $Z(1)$ and $Z(2)$ are all $N \times d$.
A decoder layer can be realised similarly. 
For this analysis, we ignore positional encodings.
We define the operations appearing in (\ref{encoder_layer}) in more detail as follows.\\

{\bf Multi-head attention:} For $h$ attention heads, define
\begin{equation}
\label{multihead_attn_defn}
\MultiHead(X) = \Concat({\head}_1, \ldots , {\head}_h) W^O
\end{equation}
with each individual attention head defined by
\begin{equation}
\label{attention_head}
{\head}_i = \Attention(Q_i , K_i , V_i) = 
\softmax
\left(
{\frac{X W^Q_i {(X W^K_i)}^T}{\sqrt{d}}}
\right)
X W^V_i
\end{equation}
for weight matrices $W^Q_i$, $W^K_i$, $W^V_i$, $W^O$.\\

{\bf Feedforward network:}
The feedforward network is of the form
\begin{equation}
\label{ffn_defn}
\FFN(X) =(\sigma(X{W_1} + {\bb}_1)) W_2  + {\bb}_2
\end{equation}
for $X$ of shape $N \times d$, $W_i$ of shape $d \times d$, and $\bb_i \in \RR^d$. 
Here $\sigma$ is an activation function such as ReLU.\\

{\bf Layer Normalisation:} Demeans and normalises each row. For a row vector $\bfxi \in \RR^d$, with mean $\mu$ and standard deviation $\sigma$, then layer normalisation transforms $\bfxi \mapsto (\bfxi - \mu)/\sigma$.\\ 

Therefore we need to analyse the effect on the Patnaik-Pearson dimension of a data manifold $X$ of each of the following operations:
\begin{itemize}
\item{Matrix products, $X \mapsto XW$, for weight matrices $W$.}
\item{Addition of matrices, $X_1 + X_2$.}
\item{Concatenation of matrices, $X_1 \oplus X_2$.}
\item{Softmax, 
$X \mapsto 
\softmax(
{\tfrac{1}{\sqrt{d}}} X
)$.
}
\item{Scaled dot-product attention, $\Attention(Q,K,V)$.}
\item{Activation functions.}
\item{Layer normalisation.}
\end{itemize}
Some of these operations generally reduce Patnaik-Pearson dimension, whilst others increase it.

\subsection{Patnaik-Pearson for Products}

We investigate the relation between $\PatnaikPearson(X)$, $\PatnaikPearson(W)$ and $\PatnaikPearson(XW)$, where $X$ is an $N \times d$ data manifold and $W$ a $d \times d$ weight matrix. 
We need the notion of two matrices being ``in general position". 

\begin{definition} Linear subspaces $V_1$ and $V_2$ of $\RR^d$, of dimensions $n_1$, $n_2$ respectively, are in general position if their intersection has the smallest possible dimension,
$$\dim( V_1 \cap V_2) = \max(0, n_1 + n_2 - d)$$
If $n_1 + n_2 \leq d$, then $V_1$ and $V_2$ do not intersect, except at the origin.
\end{definition}

\begin{definition} Matrices $A$ and $B$ are in general position if their eigenspaces are in general position - no eigenspace of $A$ overlaps with an eigenspace of $B$ more than dimension-counting requires. 
\end{definition}

Unless stated otherwise, we assume that $X$ and $W$ are in general position.
Numerical experiments (e.g. Figure \ref{fig:product_hypothesis_new}) suggest the following, for $X$ and $W$ generated using various heavy and light-tailed distributions (uniform, Cauchy, Pareto).

\begin{conjecture}
\label{pp_dim_lower_upper_bound_conjecture}
Suppose that $X$ is $N \times d$, and $W$ is $d \times m$, drawn from appropriate ensembles of matrices, with $X$ and $W$ in general position and both column-demeaned. 
Then
\begin{equation}
\label{pp_dim_lower_upper_bound_main_theorem}
{\tfrac{1}{d}} \EE(\PatnaikPearson(X)) * {\tfrac{1}{m}} \EE(\PatnaikPearson(W)) \leq
{\tfrac{1}{m}} \EE(\PatnaikPearson(XW))
\leq \min \{ {\tfrac{1}{d}} \EE(\PatnaikPearson(X)) , {\tfrac{1}{m}} \EE(\PatnaikPearson(W)) \}
\end{equation}
for sufficiently large $N$, $d$, $m$.
\end{conjecture}

Part of the conjecture is determining what ``appropriate ensembles" to consider.
We will prove (\ref{pp_dim_lower_upper_bound_main_theorem}) for the case of $X$, $W$ being elements of a so-called ``Pareto ensemble" (Theorem \ref{theorem_pp_product_pareto_ensemble}), with also $m = d$ (so $W$ is $d \times d$). 
We need some preliminary results. 

\begin{lemma} Suppose $X$ and $W$ are column-demeaned, with SVDs $X = U_X A V_X^t$, $W = U_W B V_W^t$,
where $A$, $B$ are diagonal. Then, provided $d = m$ (so $W$ is $d \times d$), 
\begin{equation}
\PatnaikPearson(XW) = \PatnaikPearson(A Q B Q^t)
\end{equation} 
\end{lemma}
{\bf Proof:} We have
$XW = U_X A (V_X^t U_W) B V_W^t = U_X A Q B Q^t (V_X^t U_W V_W^t)$, 
where $Q = V_X^t U_W$ is orthogonal $d \times d$. 
For $B$ and $Q^t$ to be dimensionally compatible, we require $m = d$.
Then, using (\ref{pp_invariant_under_orthogonals}),
$$\PatnaikPearson(XW) = \PatnaikPearson(U_X^t XW (V_X^t U_W V_W^t)^t ) = \PatnaikPearson(A Q B Q^t)$$
Simplifying further, since there are only $d$ non-zero entries of $A$, we will assume that $A$ is $d \times d$.
$\myqed$\\

\begin{theorem}
\label{product_conjecture_nu_theorem}
Suppose that $\lambda$ and $\phi$ are regularly-varying random variables, with tail exponents $\alpha$ and $\beta$ respectively, and let $d$ be a positive integer.  
Let $\lambda_1, \ldots , \lambda_d$ be iid, with distribution $\lambda$, 
and similarly $\phi_1, \ldots \phi_d$ be iid, with distribution $\phi$. 
Denote $\blambda = (\lambda_1, \ldots , \lambda_d)$, $\bphi = (\phi_1, \ldots \phi_d)$.
Let $Q$ be a randomly chosen $d \times d$ orthogonal matrix.  
Define 
$Y = 
\diag(\blambda) 
Q 
\diag(\bphi)
Q^t$, 
and consider the SVD 
$Y = U 
\diag(\bpsi) 
V^t$, where 
$\bpsi = (\psi_1, \ldots , \psi_d)$. Then, for $d$ sufficiently large,
\begin{equation}
\label{product_conjecture_nu}
{\tfrac{1}{d}} \EE(\nu(\blambda)) * {\tfrac{1}{d}} \EE(\nu(\bphi)) 
\leq 
{\tfrac{1}{d}} \EE(\nu(\bpsi)) 
\leq 
\min \{ {\tfrac{1}{d}} \EE(\nu(\blambda)), {\tfrac{1}{d}} \EE(\nu(\bphi)) \}
\end{equation}
\end{theorem}
{\bf Proof:} 
We establish the upper bound as follows.
We have
$Y_{ij} = \lambda_i (\sum_k \phi_k Q_{ik} Q_{jk})$.
From Theorem \ref{fellers_convolution_lemma}, the weighted sum  $\sum_k \phi_k Q_{ik} Q_{jk}$ will also be regularly-varying with tail exponent $\beta$, 
and from Theorem \ref{product_rv}, since $\lambda$, $\phi$ are both regularly-varying, 
their product $\lambda \phi$ is also regularly-varying, with tail exponent $\gamma = \min \{ \alpha, \beta \}$. 
Hence the $Y_{ij}$ are regularly-varying, with tail exponent $\gamma$. 
As argued by Cizeau and Bouchaud \cite{levy_matrices} 
(see also Burda et al. \cite{free_random_levy_wigner_levy}, 
Ben Arous and Guionnet \cite{spectrum_heavy_tailed}, 
and Aggarwal et al. \cite{goe_statistics_levy_matrices}), 
if the iid entries of a (symmetric) matrix follow a power law distribution
then the singular value distribution inherits the power-law tail of the entries, with the same exponent.
Hence the $\psi_i$ are drawn from a regularly-varying distribution with tail exponent $\gamma$. 
Since 
$\lim_{d \rto \infty} {\tfrac{1}{d}} \EE (\nu(\blambda)) = \nuinfty(\alpha)$, as defined in (\ref{defn_nu_infty})
and similarly for $\bphi$, $\bpsi$, and since $\nuinfty(\alpha)$ is monotonic non-decreasing in $\alpha$, it follows that
\begin{equation}
\label{nu_psi_upper_bound}
\lim_{d \rto \infty} {\tfrac{1}{d}} \EE (\nu(\bpsi)) 
= \nuinfty(\gamma) 
\leq \min \{ \nuinfty(\alpha), \nuinfty(\beta) \} 
= \lim_{d \rto \infty} \min \{ {\tfrac{1}{d}}\EE (\nu(\alpha)), {\tfrac{1}{d}} \EE (\nu(\bphi)) \} 
\end{equation}

\begin{lemma} 
\label{tr_Y_all_squared_over_tr_Ysquared_vs_nu_psi}
Given $Y$ with SVD $Y = U S V^t$, with $S = \diag(\bpsi)$, then $|\Tr(Y)| \le \Tr(S)$, and further
\begin{equation}
\label{zero_le_tr_Y_le_tr_S_le_nu_psi}
0 \le {\frac{\Tr(Y)^2}{\Tr(Y^t Y)}} \le {\frac{\Tr(S)^2}{\Tr(S^t S)}} = \nu(\bpsi)
\end{equation}
\end{lemma}
{\bf Proof:} We have $Y_{ij} = \sum_k U_{ik} \psi_k (V^t)_{kj}$, hence
$\Tr(Y) = \sum_{i,k} U_{ik} \psi_k (V^t)_{ki} = \sum_k \psi_k (V^t U)_{kk}$.
Since $U$ and $V$ are both orthogonal matrices, 
the elements $(V^t U)_{kk}$ are scalar products of unit vectors, satisfying
$-1 \le (V^t U)_{kk} \le 1$, for all $k$. 
Therefore (since $\psi_k \ge 0$ for all $k$),
$$| \Tr(Y)| \le \sum_k \psi_k | (V^t U)_{kk} | \le \sum_k \psi_k = \Tr(S)$$
Furthermore, 
$Y^t Y = (U S V^t)^t (U S V^t) = V S^t U^t U S V^t = V S^t S V^t$,
so $\Tr(Y^t Y) = \Tr(S^t S) = \sum_k \psi_k^2$.
It follows that
$0 \le {\frac{\Tr(Y)^2}{\Tr(Y^t Y)}} \le {\frac{\Tr(S)^2}{\Tr(S^t S)}} = (\sum_k \psi_k)^2 / (\sum_k \psi_k^2) = \nu(\psi)$.
$\myqed$\\

For the full proof, we use Voiculescu's free probability 
\cite{intro_random_matrices, voiculescu_limit_laws, free_random_variables}, 
specifically the notion of two elements of a non-commutative probability space being freely independent.

\begin{definition}
A non-commutative probability space $(\bbA, \tau)$ is a unital algebra $\bbA$ over $\CC$ (possibly non-commutative), with $\tau : \bbA \rightarrow \CC$ a linear functional satisfying $\tau(1) = 1$. If $\tau(xy) = \tau(yx)$ for all $x, y \in \bbA$, then $\tau$ is called a trace. 
If $\bbA$ is a Banach algebra, then $\tau$ is required to be continuous, and if $\bbA$ is also a C*-algebra, then $\tau$ is positive, in the sense that $\tau({x^{*}}x) \geq 0$ for all $x \in \bbA$. 
\end{definition}

\begin{definition} Let $\{ \bbA_j \}_{j \in I}$ be a family of subalgebras of $\bbA$, each containing the unit of $\bbA$. 
The family $\{ \bbA_j \}_{j \in I}$ is said to be {\it freely independent} if, for any positive integer $n$, and indices $k_1 \ne k_2$, $k_2 \ne k_3$, .., $k_{n-1} \ne k_n$ in $I$, and any $a_j \in \bbA_{k_j}$ ($j = 1, \ldots n$), with $\tau(a_j) = 0 \forall j$, then
$\tau( a_1 a_2 \ldots a_n ) = 0$.
Sets of elements of $\bbA$ are called {\it freely independent} if the algebras they generate are freely independent.
\end{definition}

Voiculescu proved the following:

\begin{theorem}
\label{voiculescus_theorem}
 Voiculescu \cite{voiculescu_limit_laws}.
If $A_n$, $B_n$ are sequences of (deterministic or random) matrices with limiting spectral distributions, and $U_n$ is a sequence of Haar-distributed unitaries, then $A_n$ and $U_n B_n U_n^{*}$ become asymptotically free as $n \rightarrow \infty$.
\end{theorem}

In the current setting, we are restricting ourselves to real-valued matrices (thought of as $\RR$-subalgebras of complex-valued matrices), hence unitary matrices are simply orthogonal matrices.\\

I am indebted to Dimitri Shlyakhtenko for the key ideas of Theorem \ref{dimitri_theorem} and Lemma \ref{dimitri_lemma}.

\begin{theorem}
\label{dimitri_theorem}
Let $(\bbA, \tau)$ be the non-commutative probability space consisting of $\CC$-valued $d \times d$ matrices, with 
$\tau(x) 
= {\tfrac{1}{d}} \Tr (x)$ 
for $x \in \bbA$, where $\Tr$ is the usual matrix trace. 
Let $\blambda, \bphi \in \RR^d$ be non-zero vectors of non-negative reals,
define $A = \diag(\blambda)$, $B = \diag(\bphi)$ and let
$Q$ be a randomly chosen $d \times d$ orthogonal matrix.  
Define 
$Y = A Q B Q^t$
and consider the SVD 
$Y = U 
\diag(\bpsi) 
V^t$, where 
$\bpsi = (\psi_1, \ldots , \psi_d)$, with $0 \le \psi_1 \le \ldots \le \psi_d$.
Then, asymptotically for large $d$,
\begin{equation}
\label{product_conjecture_nu_dimitri}
{\tfrac{1}{d}} \nu(\blambda) * {\tfrac{1}{d}} \nu(\bphi) 
\leq 
{\tfrac{1}{d}} \nu(\bpsi) 
\leq 
\min \{ {\tfrac{1}{d}} \nu(\blambda), {\tfrac{1}{d}} \nu(\bphi) \}
\end{equation}
\end{theorem}
{\bf Proof:}
We have $\tau(A) = {\tfrac{1}{d}} \Tr (A) = {\tfrac{1}{d}} (\sum_i \lambda_i )$.
Hence 
${\tfrac{1}{d}} \nu(\blambda) = 
{\frac
{(\tau (A))^2}
{\tau (A^t A)}
}$
 and
${\tfrac{1}{d}} \nu(\bphi) =
{\frac
{(\tau (B))^2}
{\tau (B^t B)}
}
$.
Also,
$$
{\frac
{(\tau (Y))^2}
{\tau (Y^t Y)}
}
=
{\frac
{(\tau (AQB{Q^t}))^2}
{\tau ((AQB{Q^t})^t (AQB{Q^t}))}
}
=
{\frac
{(\tau (AQB{Q^t}))^2}
{\tau( Q B^t Q^t A^t A Q B Q^t )}
}
=
{\frac
{(\tau (AQB{Q^t}))^2}
{\tau( A^2 Q B^2 Q^t )}
}
$$
Writing 
$\overline{B} = QB {Q^t}$, then ${\overline{B}}^2 = Q B^2 {Q^t}$, and
this becomes
\begin{equation}
\label{nu_psi_intermediate}
{\frac
{(\tau (Y))^2}
{\tau (Y^t Y)}
}
=
{\frac
{(\tau(A \overline{B}))^2}
{\tau((A^2 {\overline{B}}^2)}
}
\end{equation}
We need the following:
\begin{lemma}
\label{dimitri_lemma} Let $(\bbA,\tau)$ be a non-commutative probability space, and suppose that $x, y \in \bbA$ are freely independent. 
Then 
$\tau(xy) = \tau(x) \tau(y)$, and 
\begin{equation}
\label{tau_xsqrd_ysqrd}
\tau(x^2 y^2) = \tau(x^2) \tau(y)^2 + \tau(x)^2 \tau(y^2) - \tau(x)^2 \tau(y)^2
\end{equation}
\end{lemma}
{\bf Proof:} Write $x' = x - \tau(x)$ and $y' = y - \tau(y)$ (hence $\tau(x') = 0 = \tau(y')$). 
Then $x', y'$ are freely independent, so for any word
$w = {x'}^{i_1} {y'}^{j_1} \ldots {x'}^{i_k} {y'}^{j_k}$, with $\sum_{r=1}^k i_r > 0$, and $\sum_{r=1}^k j_r > 0$, 
then $\tau(w) = 0$. 
Using this, first of all
$$0 = \tau(x' y' ) = \tau((x - \tau(x))(y - \tau(y))) = \tau(xy) - \tau(x) \tau(y)$$
Hence $\tau(xy) = \tau(x)\tau(y)$. 
For (\ref{tau_xsqrd_ysqrd})  
first of all note that 
$0 = \tau((x')^2) = \tau(x^2) - \tau(x)^2$, and further
$0 = \tau( (x')^2 y' ) = \tau(x^2 y) - \tau(x^2) \tau(y)$, so $\tau(x^2 y) = \tau(x^2) \tau(y)$. 
Similarly $\tau(x y^2) = \tau(x) \tau(y^2)$.
Then write 
$\tau(x^2 y^2) = \tau( (x' + \tau(x))^2 (y' + \tau(y))^2 )$, expand and simplify.
$\myqed$\\

By Theorem \ref{voiculescus_theorem},
$A$ and $\overline{B}$ are asymptotically free for large $d$.
Assuming freeness, then 
$\tau(A \overline{B}) = \tau(A) \tau(\overline{B}) = \tau(A) \tau(B)$, 
since 
$\Tr({\overline{B}}^p) = \Tr(B^p)$ for any positive integer $p$.
Using (\ref{tau_xsqrd_ysqrd}),
$$\tau(A^2 {\overline{B}}^2 ) 
= \tau(A^2) \tau(\overline{B})^2 + \tau(A)^2 \tau({\overline{B}}^2) - \tau(A)^2 \tau(\overline{B})^2
= \tau(A^2) \tau(B)^2 + \tau(A)^2 \tau(B^2) - \tau(A)^2 \tau(B)^2
$$
Hence (\ref{nu_psi_intermediate}) becomes
\begin{equation}
\label{nu_psi_nu_lambda_nu_phi}
{\frac
{(\tau (Y))^2}
{\tau (Y^t Y)}
}
=
{
\frac
{\tau(A)^2 \tau(B)^2}
{\tau(A^2) \tau(B)^2 + \tau(A)^2 \tau(B^2) - \tau(A)^2 \tau(B)^2}
}
=
{\frac
{
{\tfrac{1}{d}} \nu( \blambda) * {\tfrac{1}{d}} \nu( \bphi)
}
{
{\tfrac{1}{d}} \nu( \blambda) + {\tfrac{1}{d}} \nu( \bphi) - ({\tfrac{1}{d}} \nu( \blambda) * {\tfrac{1}{d}} \nu( \bphi))
}
}
\end{equation}
Since $0 \le {\tfrac{1}{d}} \nu( \blambda),  {\tfrac{1}{d}} \nu( \bphi) \le 1$, then finding upper and lower bounds for 
${\tfrac{1}{d}} \nu(\bpsi)$ reduces to finding upper and lower bounds for the function
$f(s,t) := {\frac{st}{s + t - st}}$ on the unit square $[0,1]^2$.
It is immediate that
$0 \le f(s,t) \le 1$, with the lower and upper bounds being attained only at $(0,0)$ and $(1,1)$ respectively.
Writing 
$s + t - st = (1-t)s + t.1 = (1-s)t + s.1$
it follows that
$$\max\{s,t\} \le s + t - st \le 1 \quad \implies \quad
st \le f(s,t) \le {\frac{st}{\max \{s,t\}}} = \min\{s,t\}$$
As a result, 
\begin{equation}
\label{nearly_done}
{\tfrac{1}{d}} \nu( \blambda) * {\tfrac{1}{d}} \nu( \bphi) 
\le 
{\frac
{(\tau (Y))^2}
{\tau (Y^t Y)}
}
\le
\min \{	 {\tfrac{1}{d}} \nu( \blambda) , {\tfrac{1}{d}} \nu( \bphi) \}
\end{equation}
From (\ref{zero_le_tr_Y_le_tr_S_le_nu_psi}), 
${\frac
{(\tau (Y))^2}
{\tau (Y^t Y)}
}
\le \nu(\bpsi)$, so for large $d$,
${\tfrac{1}{d}} \nu( \blambda) * {\tfrac{1}{d}} \nu( \bphi) \le {\tfrac{1}{d}} \nu(\bpsi)$.
The upper bound 
${\tfrac{1}{d}} \nu(\bpsi) \le \min \{ {\tfrac{1}{d}} \nu( \blambda) , {\tfrac{1}{d}} \nu( \bphi) \}$ 
follows from the argument of (\ref{nu_psi_upper_bound}). 
This completes the proof of Theorem \ref{dimitri_theorem}. 
$\myqed$\\

To apply the freeness argument above to matrices, we need finite moments of all orders, which we do not have for our heavy-tailed random variables. 
To bridge this gap we work with truncated Pareto distributions, defined as follows. 
For each $K > 1$, define $X_K$ to be the random variable with pdf
$$f_K(t) = {\frac{f(t)}{F(K)}}, \, 1 \le t \le K, \quad f_K (t) = 0 \, \otherwise$$
where $f(t)$ was defined in (\ref{pareto_alpha_pdf_X}) by $f(t) = {\frac{\alpha}{t^{\alpha+1}}}$, 
and 
$F(x) = \int_1^x f(t) dt$.
It follows that
$$\EE (X_K) = 
K 
\left(  
{\frac{K^{\alpha - 1} - 1}{K^\alpha - 1}}
\right)
\left(  
{\frac{\alpha}{\alpha - 1}}
\right),
\quad
\EE (X_K^2) = 
K^2 
\left(  
{\frac{K^{\alpha - 2} - 1}{K^\alpha - 1}}
\right)
\left(  
{\frac{\alpha}{\alpha - 2}}
\right),
\quad \alpha \ne 1, 2
$$
which are both finite for all $\alpha$
and $K > 1$. 
Define
\begin{equation}
\label{sigma_alpha_K}
\sigma(\alpha,K) := 
{\frac{\EE (X_K)^2}{\EE (X_K^2)}}
=
{\frac{\alpha (\alpha - 2)}{(\alpha - 1)^2}}
{\frac{(K^{\alpha - 1} - 1)^2}{(K^\alpha - 1)(K^{\alpha - 2} -1)}}
= 
C(\alpha) {\frac{(K^{\alpha - 1} - 1)^2}{(K^\alpha - 1)(K^{\alpha - 2} -1)}}
\end{equation}
using the notation of (\ref{defn_s_C_alpha_nu_alpha_d}). 
Further, for $\alpha = 1, 2$, then
\begin{equation}
\label{sigma_alpha_K_alpha_1_2}
\sigma(1,K) 
=
{\frac{K ( \ln(K))^2}{(K-1)^2}}
, \quad
\sigma(2, K) 
= 
{\frac{2(K-1)}{(K+1)\ln(K)}}
\end{equation}

\begin{lemma} For all $\alpha > 0$, 
$\lim_{K \rto \infty} {\frac{\EE (X_K)^2}{\EE (X_K^2)}}
= \nuinfty (\alpha)$, as defined in (\ref{defn_nu_infty}).
\end{lemma}
{\bf Proof:}  For $\alpha > 2$, $\sigma(\alpha, K) 
\approx 
C(\alpha) 
{\frac
{K^{2(\alpha - 1)}}
{K^{2\alpha - 2}}
} 
= C(\alpha)$ 
for large $K$.
For $\alpha \leq 2$, it is immediate from (\ref{sigma_alpha_K}, \ref{sigma_alpha_K_alpha_1_2}) that $\sigma(\alpha, K) \rto 0$ as $K \rto \infty$. 
The result follows. $\myqed$\\

We can now complete the proof of Theorem \ref{product_conjecture_nu_theorem}.
Suppose that $\lambda_K$ and $\phi_K$ are the truncations at $K$ of regularly-varying random variables $\lambda$, $\phi$, with tail exponents $\alpha$ and $\beta$ respectively.

\begin{lemma}
\label{limit_nu_lambdaK_eq_nu_lambda} 
For fixed $d$,  
$\lim_{K \rto \infty} \EE (\nu (\blambda_K ) ) = \EE (\nu (\blambda)) = \nu(\alpha,d)$
\end{lemma}
{\bf Proof:}
Suppose we have independent draws $\lambda_1, \ldots \lambda_d$ from $\lambda$. 
Since $\PP(\lambda \le K) = 1 - {\tfrac{1}{K^\alpha}}$, then
$$\PP(\lambda_i \le K \, \forall \, 1 \le i \le d) = (1 - {\tfrac{1}{K^\alpha}})^d =: \pdka$$
We can write
$$
\EE(\nu(\blambda)) 
= 
\pdka \EE ( \nu(\blambda) \, | \, \lambda_i \le K \, \forall \, i) 
+
(1 - \pdka) \EE ( \nu(\blambda) \, | \, \exists \, i \, s.t. \, \lambda_i > K)
$$ 
Now, 
$
\EE ( \nu(\blambda) \, | \, \lambda_i \le K \, \forall \, i) 
= \EE( \nu(\blambda_K))
$, 
whereas for the second term, if there exists $i$ such that $\lambda_i > K$, 
all we can say about $\nu(\blambda)$ is that 
$1 \le \nu(\blambda) \le d$. Hence
$$
\pdka \EE( \nu(\blambdaK)) + (1 - \pdka) 
\le 
\EE(\nu(\blambda)) 
\le 
\pdka \EE( \nu(\blambdaK)) + (1 - \pdka) d
$$
As $K \rto \infty$, then $\pdka \rto 1$, and $(1 - \pdka)$ and $(1 - \pdka) d$ vanish. 
The result follows. 
$\myqed$\\

Continuing this argument, given independent draws 
$\lambda_1, \ldots ,\lambda_d$ from $\lambda$, 
and $\phi_1, \ldots , \phi_d$ from $\phi$,
then
$$
\PP(\lambda_i, \phi_i \le K \, \forall i) 
= (1 - {\tfrac{1}{K^\alpha}})^d (1 - {\tfrac{1}{K^\beta}})^d 
=: \pdkab
$$
$$\implies \quad \EE(\nu (\bpsi )) 
= \pdkab \EE( \nu (\bpsi ) \, | \, \lambda_i, \phi_i \le K \, \forall \, i )
+ (1 - \pdkab) 
\EE( \nu(\bpsi) 
\, | \, 
\exists 
\, i \, 
s.t. 
\, 
\lambda_i > K 
\, or \, 
\phi_i > K )
$$
The first term is determined by (\ref{product_conjecture_nu_dimitri}), namely:
$$
\EE 
\left( 
{\tfrac{1}{d}} 
\nu(\blambdaK) 
* 
{\tfrac{1}{d}} 
\nu(\bphiK) 
\right)
\leq 
{\tfrac{1}{d}} \EE( \nu (\bpsi ) \, | \, \lambda_i, \phi_i \le K \, \forall \, i )
\leq 
\EE \left( \min \{ {\tfrac{1}{d}} \nu(\blambdaK), {\tfrac{1}{d}} \nu(\bphiK) \} \right)
$$
$$
\leq 
\min \{ {\tfrac{1}{d}} \EE(\nu(\blambdaK)), {\tfrac{1}{d}} \EE(\nu(\bphiK)) \}
$$
Using independence of $\lambda$ and $\phi$ this becomes:
\begin{equation}
\label{expectation_nu_psi_K}
{\tfrac{1}{d}} \EE (\nu(\blambdaK)) * {\tfrac{1}{d}} \EE(\nu(\bphiK))
\leq 
{\tfrac{1}{d}} \EE(\nu(\bpsi ) \, | \, \lambda_i, \phi_i \le K \, \forall \, i )
\leq 
 \min \{ {\tfrac{1}{d}} \EE(\nu(\blambdaK)), {\tfrac{1}{d}} \EE(\nu(\bphiK)) \}
\end{equation}
The second term is controlled by the boundedness of $\nu(\bpsi)$:
$$
1 \le 
\EE( \nu(\bpsi) 
\, | \, 
\exists 
\, i \, 
s.t. 
\, 
\lambda_i > K 
\, or \,
\phi_i > K )
\le d
$$
Combining these two, it follows that
\begin{equation}
\label{product_conjecture_nu_almost_there}
\begin{split}
& \pdkab * {\tfrac{1}{d}} \EE (\nu(\blambdaK)) * {\tfrac{1}{d}} \EE(\nu(\bphiK)) 
+
(1 - \pdkab)\\
& \le \EE( \nu(\bpsi) )\\
& \le 
\pdkab * {\tfrac{1}{d}} \EE( \min \{ \nu(\blambdaK)), \nu(\bphiK) \}) 
+
(1 - \pdkab) d
\end{split}
\end{equation}
In the same way as above, as $K \rto \infty$, then $\pdkab \rto 1$, both $(1 - \pdkab)$ and $(1 - \pdkab) d$ vanish, and 
by Lemma \ref{limit_nu_lambdaK_eq_nu_lambda}, 
$\EE (\nu(\blambdaK)) 
\rto 
\EE 
(
\nu(\blambda)
)$, 
and 
$\EE(\nu(\bphiK)) \rto \EE(\nu(\bphi))$.
Thus (\ref{product_conjecture_nu_almost_there}) reduces to (\ref{product_conjecture_nu}).
This completes the proof of Theorem \ref{product_conjecture_nu_theorem}.
$\myqed$\\

As a corollary of this, we have proved Conjecture \ref{pp_dim_lower_upper_bound_conjecture}
in a specific case.

\begin{definition}
For positive integers $N$, $d$ with $N \ge d$, and real $\alpha > 0$, we define the Pareto ensemble $P(\alpha, N,d)$ to be the collection of 
real-valued random matrices $X$ of the form $X = U S V^t$, where $U$ is $N \times N$, $S$ is $N \times d$ and $V$ is $d \times d$; $U$ and $V$ are orthogonal, and $S$ is diagonal with $S_{ii} = \lambda_i$, $1 \le i \le \min(N,d)$, and $\lambda_i$ are independent draws from $\lambda$, which is Pareto with tail exponent $\alpha$. 
This implicitly includes the light-tailed case where the distribution of $\lambda$ is bounded ($\alpha = \infty$). 
\end{definition}
Note that this resembles the Wigner-Levy ensemble defined by Cizeau and Bouchaud in \cite{levy_matrices}. 

\begin{theorem}
\label{theorem_pp_product_pareto_ensemble}
For column-demeaned $X \in P(\alpha, N,d)$, $W \in P(\beta, d,d)$, then 
for sufficiently large $N$, $d$, 
\begin{equation}
{\tfrac{1}{d}} \EE(\PatnaikPearson(X)) * {\tfrac{1}{d}} \EE(\PatnaikPearson(W)) \leq
{\tfrac{1}{d}} \EE(\PatnaikPearson(XW))
\leq  \min \{ {\tfrac{1}{d}} \EE(\PatnaikPearson(X)), {\tfrac{1}{d}} \EE(\PatnaikPearson(W)) \}
\end{equation}
\end{theorem}
{\bf Proof:} This is immediate from Theorem \ref{product_conjecture_nu_theorem} $\myqed$.\\

Theorem \ref{theorem_pp_product_pareto_ensemble} relies on $X$ and $W$ being in general position relative to one another, so in particular it needs to be modified for the product $X^T X$, as shown in Figure \ref{fig:nu_over_d_X_nu_over_d_XTX_nu_over_d_XT_times_nu_over_d_X_pareto.pdf}. In this case we have:

\begin{lemma} Suppose that $X$ is $N \times d$, with $\PatnaikPearson(X) = \nu(\blambda)$, for $\blambda \in \RR^d$ arising from the SVD decomposition of $X$. 
Then 
$\PatnaikPearson(X^T X)
= \nu (\blambda^2)$.
\end{lemma}
{\bf Proof:} Following (\ref{pp_int_dim}) and assuming that $X = \Xcoldem$, we start with $X = US V^T$, where $S$ is $N \times d$, with diagonal elements $\lambda_i$, $1 \leq i \leq d$, and $\blambda = (\lambda_i) \in \RR^d$, then 
$$X^T X = ( U S V^T)^T (U S V^T) = V (S^T S) V^T$$
and $S^T S$ is $d \times d$ with diagonal elements $\lambda_i^2$. The result follows. $\myqed$

\begin{figure}
    \centering
    \includegraphics[width=0.45\textwidth]{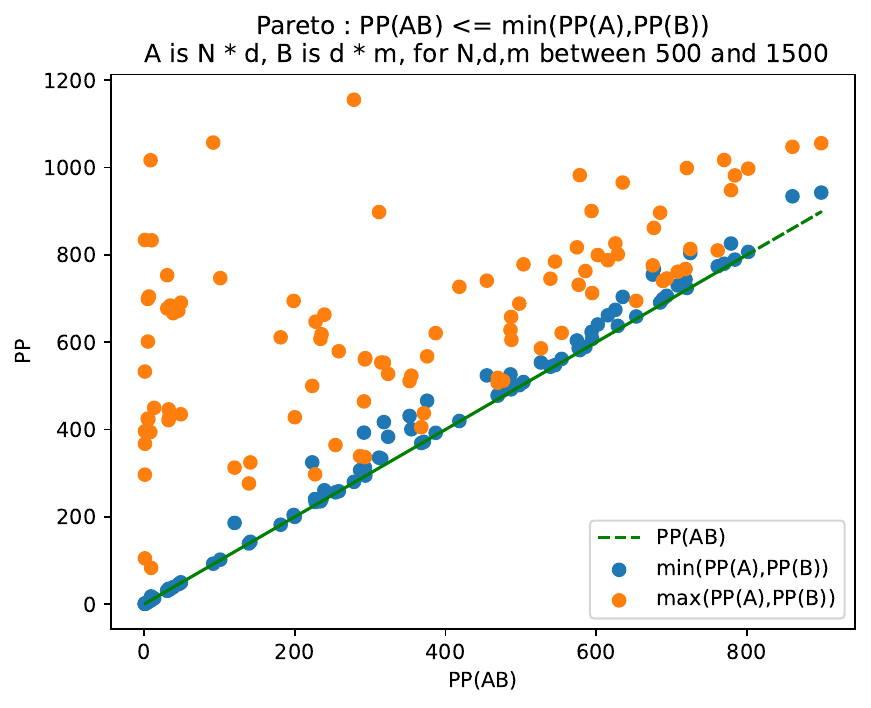} 
    \includegraphics[width=0.45\textwidth]{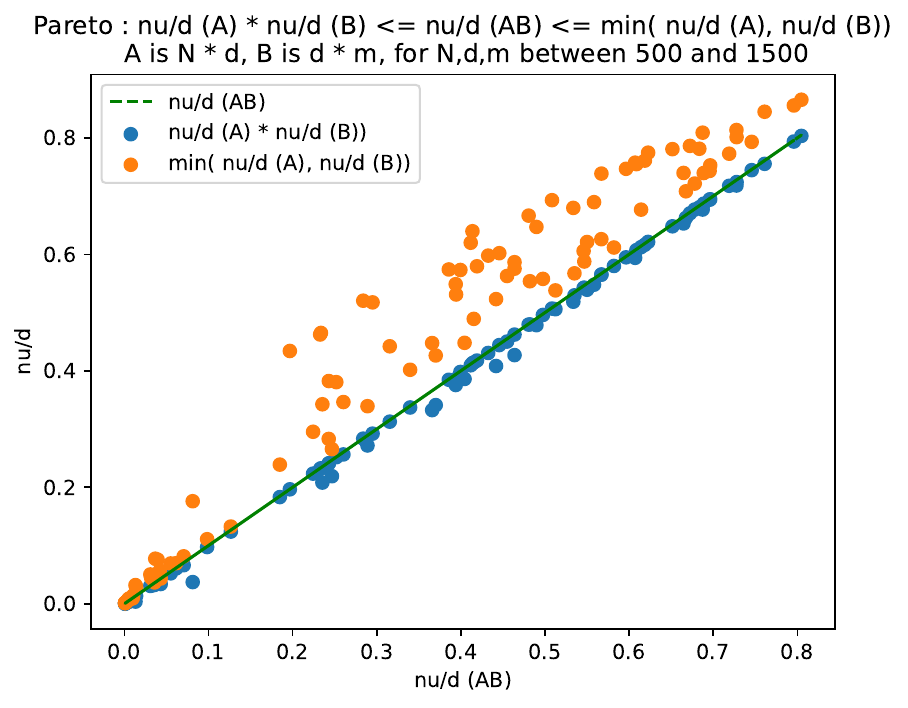}
    \caption{Product hypotheses : Conjecture \ref{pp_dim_lower_upper_bound_conjecture}, Theorem \ref{theorem_pp_product_pareto_ensemble}, for randomly sampled $\alpha$, $\beta$. 
	} 
    \label{fig:product_hypothesis_new}
\end{figure}

\begin{figure}
    \centering
    \includegraphics[width=0.45\textwidth]{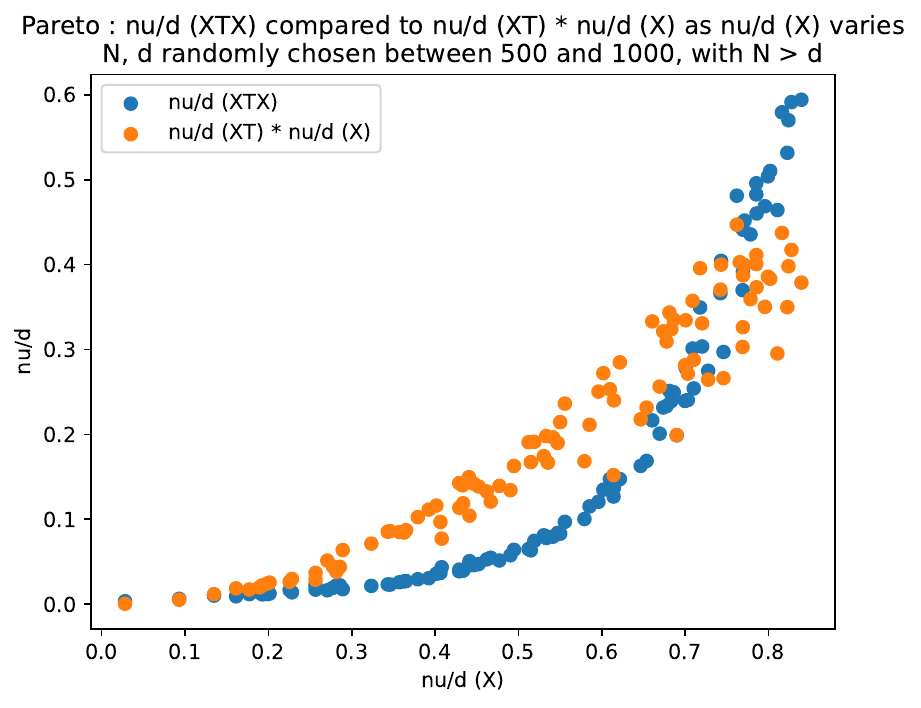} 
    \includegraphics[width=0.45\textwidth]{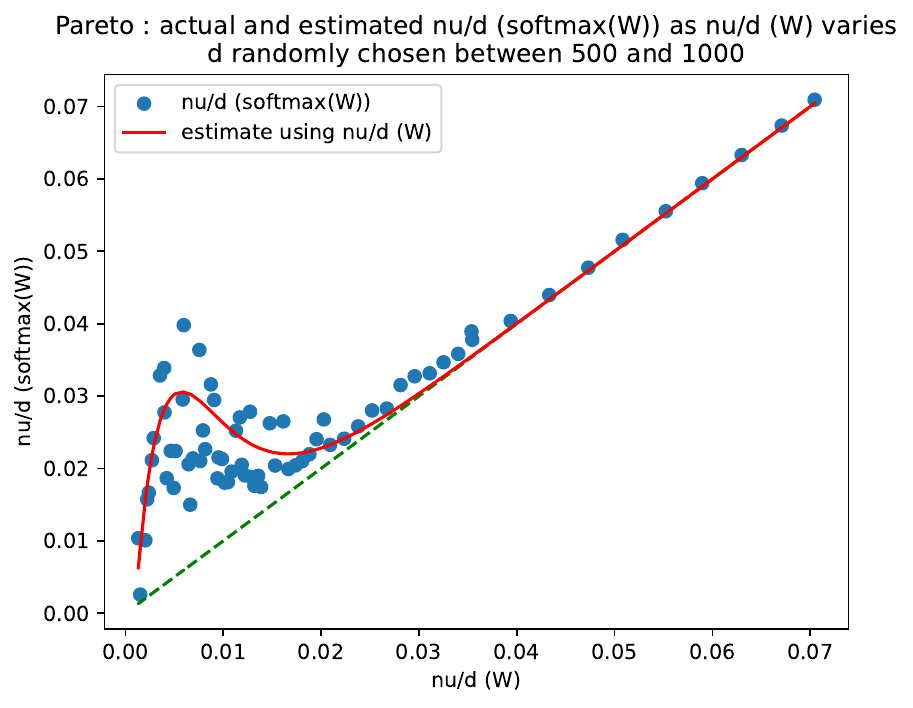} 
    \caption{(a) Product hypothesis (Conjecture \ref{pp_dim_lower_upper_bound_conjecture}) is not satisfied by $X^T X$. (b) $\PatnaikPearson(\softmax(W))$ as a function of $\PatnaikPearson(W)$, using (\ref{pp_dim_softmax_approx}, \ref{defn_sf}).} 
    \label{fig:nu_over_d_X_nu_over_d_XTX_nu_over_d_XT_times_nu_over_d_X_pareto.pdf}
	\label{fig:nu_over_d_softmaxW_actual_vs_estimated_as_nu_over_d_W_varies_pareto.pdf}
\end{figure}

\subsection{Softmax}

In general there is no clear relation between $\PatnaikPearson(\softmax(X))$ and $\PatnaikPearson(X)$, because we have defined $\PatnaikPearson(X)$ to be invariant under rotations, translations and positive scalings, whereas $\softmax(X)$ has no such invariance. However in ``typical" cases we present a close approximation.

\begin{definition} For an $N \times d$ data manifold $X$, define $\softmax(X)$ to be the row-wise softmax with inverse temperature $\beta > 0$, namely
\begin{equation}
\label{def_softmax}
{\softmax(X)}_{ij} = \exp(\beta X_{ij}) / \sum_{k=1}^d \exp(\beta X_{ik})
\end{equation}
\end{definition}

In the limit as $\beta \rto 0^{+}$, then all entries tend to ${\tfrac{1}{d}}$. 
So for small $\beta$,  $\softmax(X) = {\tfrac{1}{d}} {\bf 1}_{N \times d} + \Delta$, where $\Delta$ is ``random noise" of shape $N \times d$ with individual terms tending to zero. 
However, by the definition of Patnaik-Pearson dimension, the row-wise demeaning procedure removes the ${\bf 1}_{N \times d}$ term and $\PatnaikPearson(\softmax(X)) = \PatnaikPearson(\Delta)$, which can take arbitrary values in the range $[1,d]$. 
 In the limit as $\beta \rto \infty$, then $\softmax(X)$ converges to a matrix with exactly one non-zero entry (of 1) in each row, in position (for row $i$) given by $\argmax_k \{ X_{ik} \}$.
For the case 
$\beta = 
{\tfrac{1}{\sqrt{d}}}$, 
and  $N, d = O(1000)$, then $\PatnaikPearson(\softmax(X))$ is well-approximated by 
\begin{equation}
\label{pp_dim_softmax_approx}
\PatnaikPearson(\softmax(X)) \approx \PatnaikPearson(X) + \sff ( {\tfrac{1}{d}} \PatnaikPearson(X) )
\end{equation}
where 
\begin{equation}
\label{defn_sf}
\sff (x) = 22.0 * (x - 0.001) * \exp( -250 x)
\end{equation}
See Figure \ref{fig:nu_over_d_softmaxW_actual_vs_estimated_as_nu_over_d_W_varies_pareto.pdf}.
For small values of $\PatnaikPearson(X)$, softmax increases Patnaik-Pearson dimension, 
whereas for larger values of $\PatnaikPearson(X)$, we see that $\PatnaikPearson(\softmax(X)) \approx \PatnaikPearson(X)$.

\subsection{Attention}

Following the notation of \cite{complete_mot}, we define scaled dot product attention as follows. 
For $X$ of shape $N \times d$, and weight matrices  $W^Q$, $W^K$,$W^V$ all of shape $d \times d$, define 
\begin{equation}
\label{attention_defn}
Q = X W^Q, \quad K = X W^K, \quad V = X W^V, \quad
\Attention(Q, K, V) = 
\softmax
\left(
{\frac{Q K^T}{\sqrt{d}}}
\right)
V
\end{equation}
where softmax is applied row-wise, as in (\ref{def_softmax}) (with $\beta = 1$ here). 
For our analysis of the Patnaik-Pearson dimension, we build up to this step by step. 
Given $X$, $W$ of shape $N \times d$ and $d \times d$ respectively, the Patnaik-Pearson dimension for $XW X^T$ is well-approximated using (\ref{pp_dim_lower_upper_bound_main_theorem}) 
(see Figure \ref{fig:nu_over_d_XWXT_actual_vs_estimate_pareto.pdf})
$$
{\tfrac{1}{N}} \PatnaikPearson(XW X^T) 
\approx 
{\tfrac{1}{d}} \PatnaikPearson(X) 
* 
{\tfrac{1}{d}} \PatnaikPearson(W) 
* 
{\tfrac{1}{N}} \PatnaikPearson(X^T)
$$
Numerical experiments choosing $N$ and $d$ randomly between 500 and 1000, with $N > d$, randomly choosing the Patnaik-Pearson dimension (between 1 and $d$) for $X$, $W^Q$, $W^K$ and $W^V$,  and calculating $\Attention(Q, K, V)$ 
show reasonable agreement between 
$\PatnaikPearson(\Attention(Q, K, V))$ and the estimate given by combining (\ref{pp_dim_lower_upper_bound_main_theorem}, \ref{pp_dim_softmax_approx}, \ref{defn_sf}) - see Figure \ref{fig:nu_over_d_XWXT_actual_vs_estimate_pareto.pdf}.

\begin{figure}
    \centering
    \includegraphics[width=0.45\textwidth]{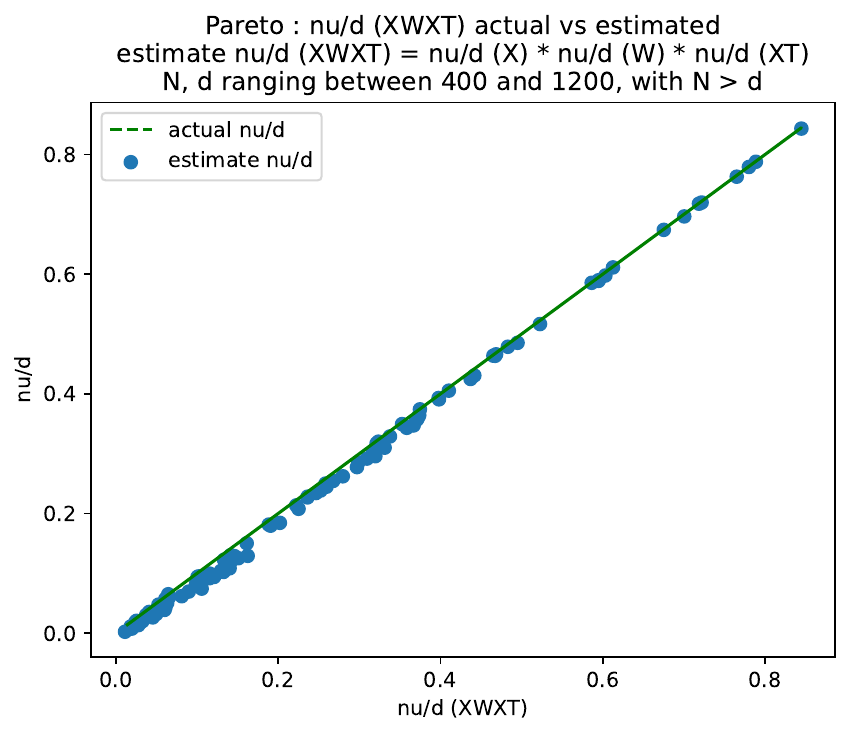} 
    \includegraphics[width=0.45\textwidth]{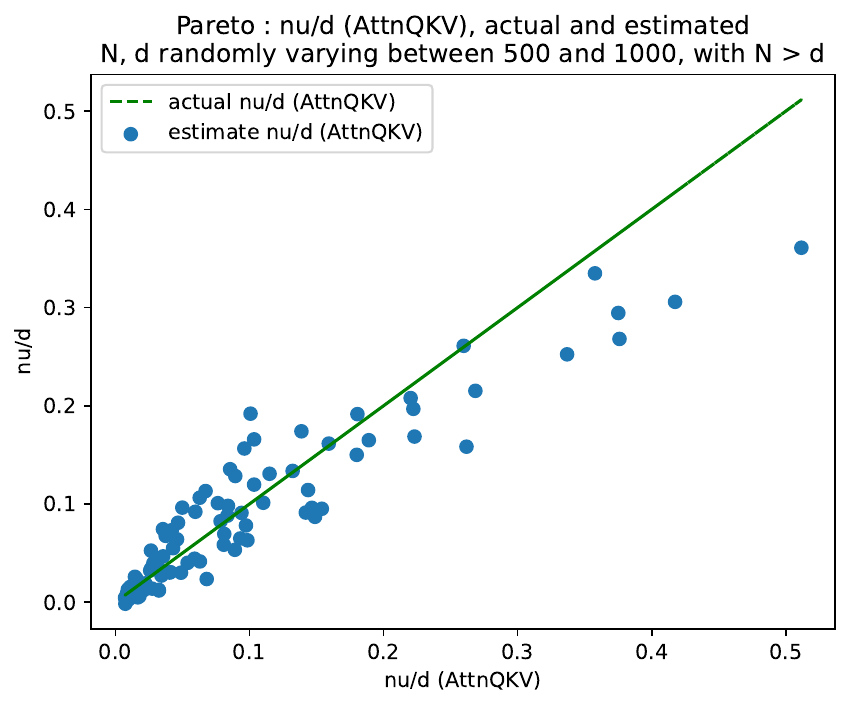} 
    \caption{(a) Patnaik-Pearson dimension for $X W X^T$ using Theorem \ref{theorem_pp_product_pareto_ensemble}. (b) Patnaik-Pearson dimension of $\Attention(Q, K, V)$, using (\ref{pp_dim_softmax_approx}, \ref{defn_sf}).} 
    \label{fig:nu_over_d_XWXT_actual_vs_estimate_pareto.pdf}
\end{figure}

\subsection{Activation functions : ReLU}

We investigate the effect of the ReLU 
activation function.
For $X$ an $N \times d$ data manifold, define
\begin{equation}
\label{defn_relu}
(\ReLU(X))_{ij} = 
\left\{
\begin{aligned}
X_{ij} \, : \, X_{ij} \geq 0 \\
0 \, : \, X_{ij} < 0
\end{aligned}
\right.,
\end{equation}
Figure \ref{fig:ReLU} shows the effect of ReLU on Patnaik-Pearson dimension.
A good approximation, at least for $N, d = O(1000)$, is
\begin{equation}
\label{relu_approx}
{\tfrac{1}{d}} \PatnaikPearson(\ReLU(X)) = a * ( {\tfrac{1}{d}} \PatnaikPearson(X))^b
\end{equation}
with $a = \exp(-0.05) = 0.891$, and $b = 0.69$.

\begin{figure}
    \centering
    \includegraphics[width=0.45\textwidth]{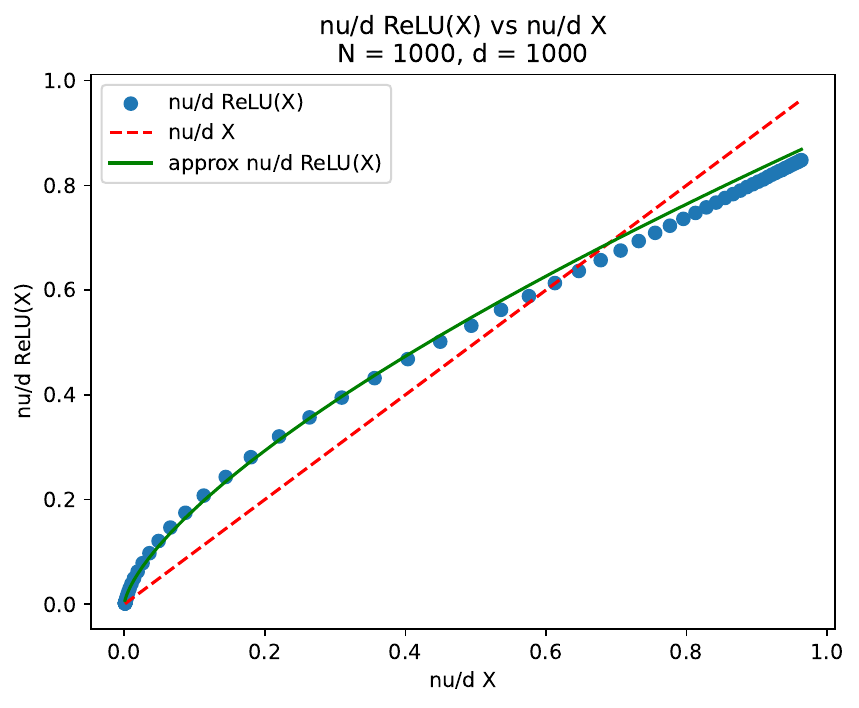} 
    \includegraphics[width=0.45\textwidth]{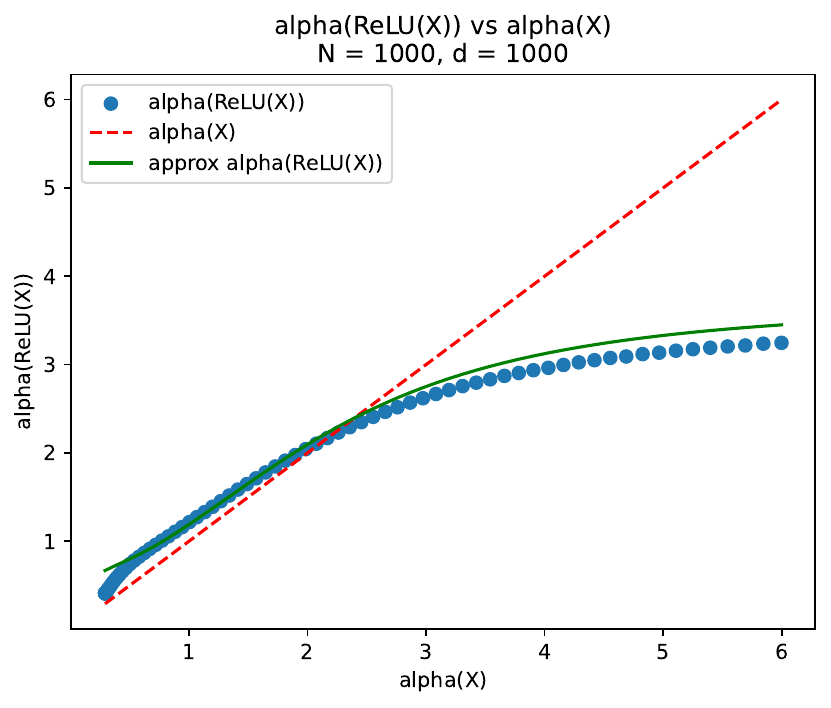} 
    \caption{The effect of ReLU on Patnaik-Pearson dimension. Approximation given by (\ref{relu_approx}).} 
    \label{fig:ReLU}
\end{figure}

\subsection{Addition, Interpolation and Concatenation}

For $X_1$ and $X_2$ both $N \times d$, define $X_1 + X_2$ as the usual matrix addition.
As shown in Figure \ref{fig:pp_dim_addition}, $\PatnaikPearson(X_1 + X_2)$ skews heavily towards 
$\min \{ \PatnaikPearson(X_1),\PatnaikPearson(X_2) \}$. 
As expected by Theorem \ref{fellers_convolution_lemma}, the heavier-tailed (smaller Patnaik-Pearson dimension) distribution dominates the lighter-tailed distribution.

\begin{figure}
    \centering
    \includegraphics[width=0.45\textwidth]{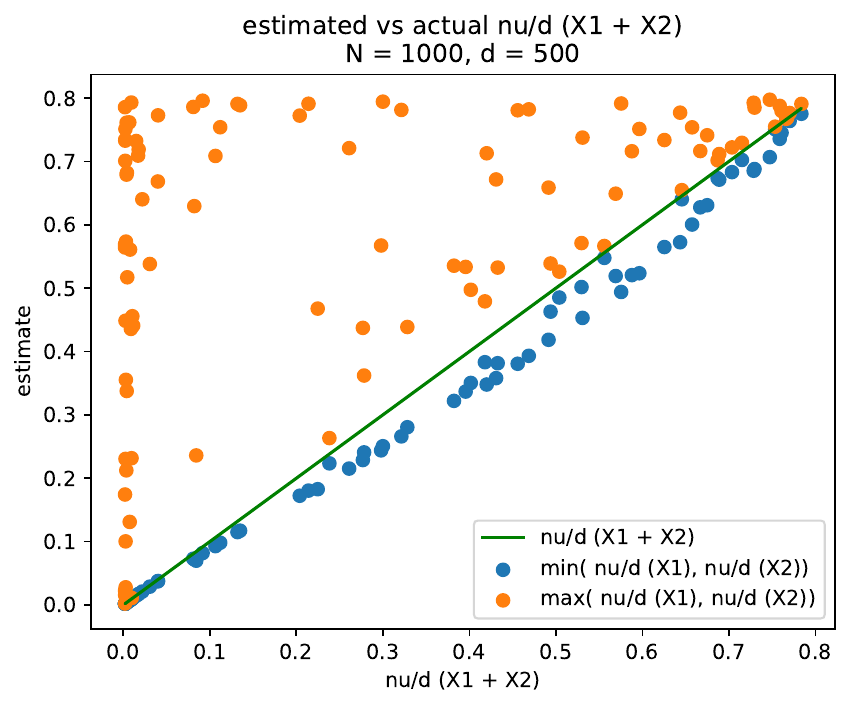} 
    \includegraphics[width=0.45\textwidth]{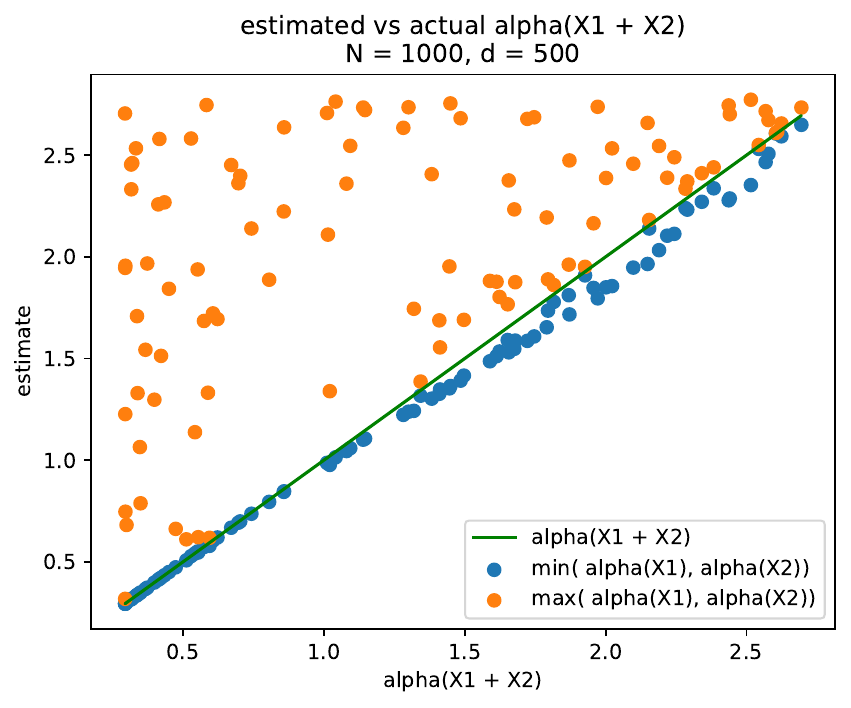} 
    \caption{Sum of two data manifolds $X_1 + X_2$ : heavy tails dominate.} 
    \label{fig:pp_dim_addition}
\end{figure}

For $X_0$ and $X_1$ both $N \times d$, and $0 \leq t \leq 1$, define $X_t$ by
\begin{equation}
\label{defn_interpolation}
(X_t)_{ij} = (1 - t) * (X_0)_{ij} + t * (X_1)_{ij}
\end{equation}
As illustrated by the example in Figure \ref{fig:interpolation_alpha_Xt_alpha_X0_0806_alpha_X1_2275_as_t_varies.pdf}, 
if $\PatnaikPearson(X_0) \leq \PatnaikPearson(X_1)$, then $\PatnaikPearson(X_t)$ remains close to $\PatnaikPearson(X_0)$ until $t$ is close to 1, before a final rapid transition towards $\PatnaikPearson(X_1)$.
This is as expected by Theorem \ref{fellers_convolution_lemma}, in that the heavier-tailed distribution dominates the lighter-tailed distribution.

\begin{figure}
    \centering
    \includegraphics[width=0.45\textwidth]{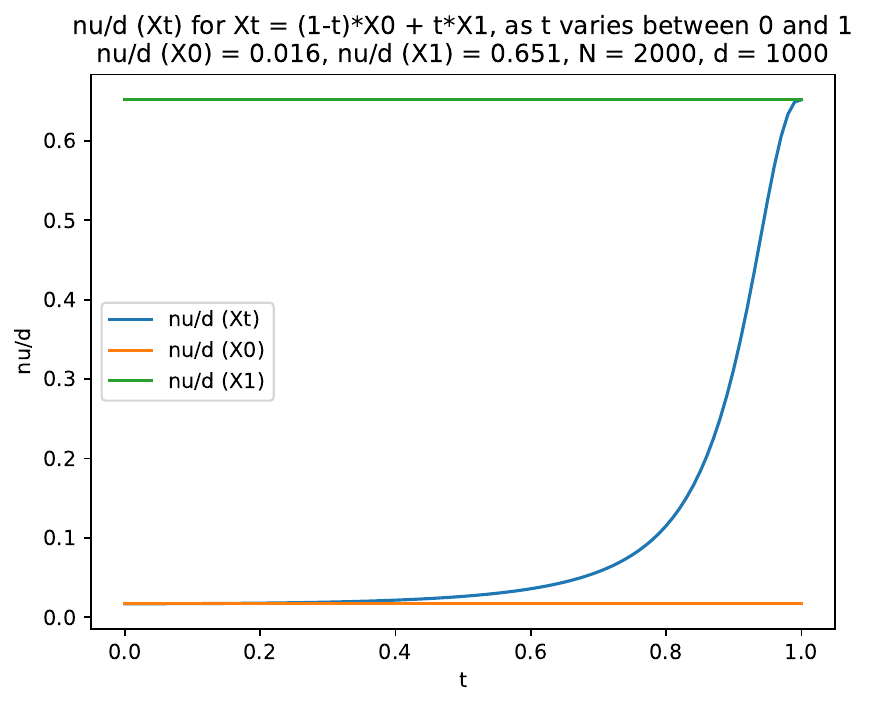} 
    \includegraphics[width=0.45\textwidth]{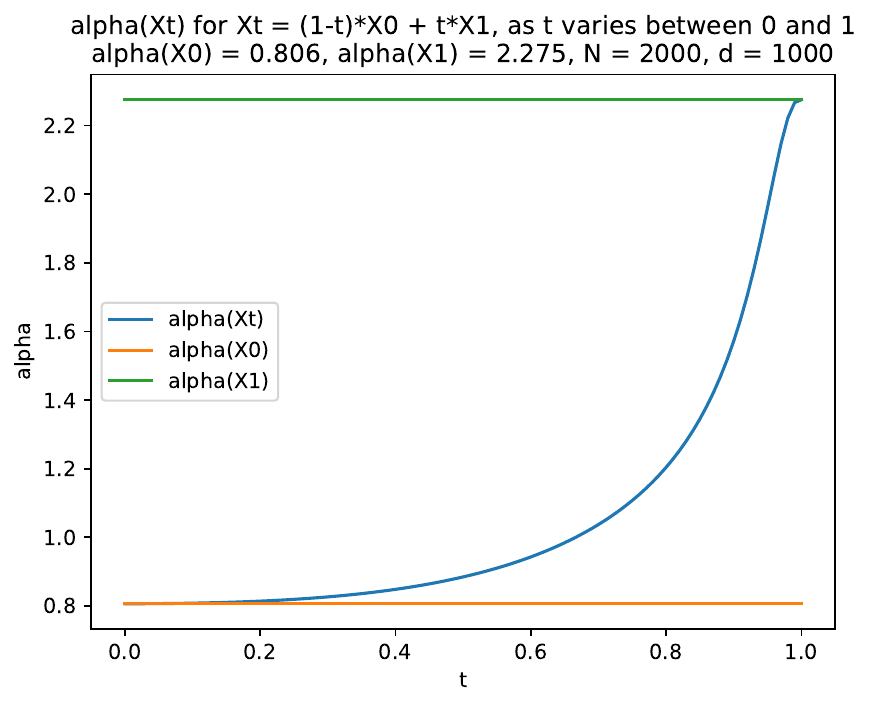} 
    \caption{Interpolation between data manifolds $X_0$ and $X_1$. Heavier tails dominate.} 
    \label{fig:interpolation_alpha_Xt_alpha_X0_0806_alpha_X1_2275_as_t_varies.pdf}
\end{figure}

Given $X_1$ and $X_2$ of dimensions $N_1 \times d$ and $N_2 \times d$ respectively, define the concatenation $X_1 \oplus X_2$ to be the $(N_1 + N_2) \times d$ matrix with entries
\begin{equation}
\label{defn_concat}
(X_1 \oplus X_2)_{ij} =  \left\{
\begin{aligned}
(X_1)_{ij} \, : \, i \leq N_1 \\
(X_2)_{(i - N_1) j} \, : \, N_1 < i \leq N_1 + N_2
\end{aligned}
\right.
\end{equation}
It is apparent from Figure \ref{fig:X1_concat_X2_cauchy.pdf} that 
\begin{equation}
\label{pp_dim_concat}
\PatnaikPearson( X_1 \oplus X_2 ) \approx \min \{ \PatnaikPearson( X_1), \PatnaikPearson( X_2 ) \}
\end{equation}
Again, the heavier-tailed distribution dominates.

\begin{figure}
    \centering
    \includegraphics[width=0.45\textwidth]{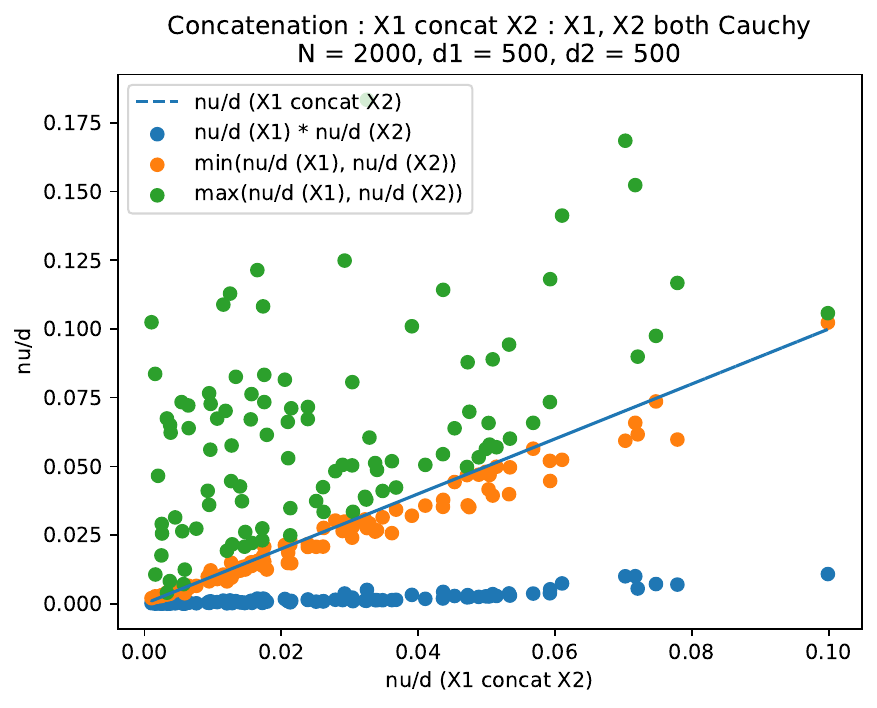}
    \includegraphics[width=0.45\textwidth]{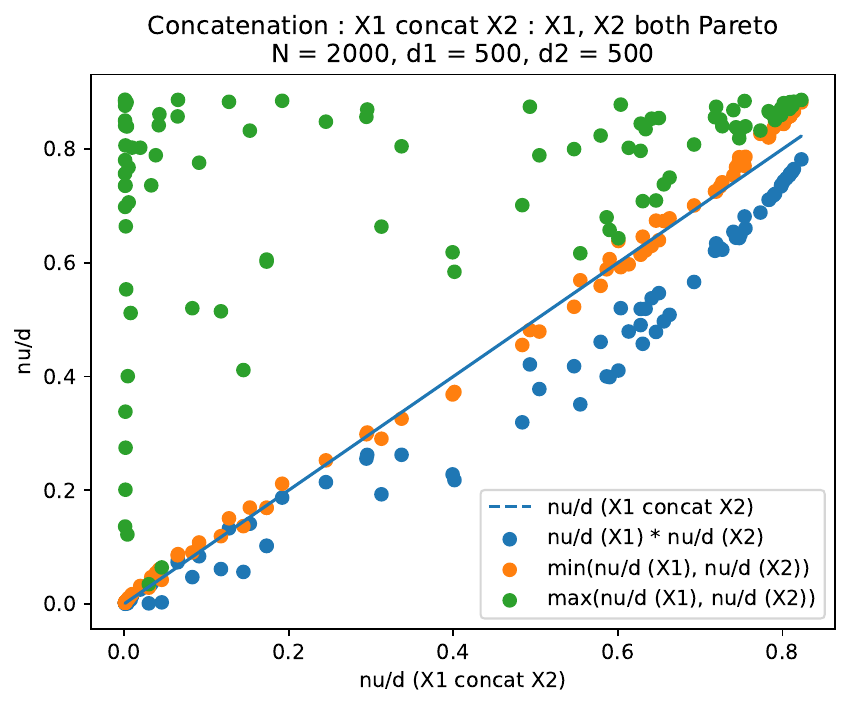}
    \caption{Concatenation : Patnaik-Pearson dimension of $X_1 \oplus X_2$, illustrating (\ref{pp_dim_concat}). Heavier tails dominate.} 
    \label{fig:X1_concat_X2_cauchy.pdf}
	\label{fig:X1_concat_X2_pareto.pdf}
	\label{fig:X1_concat_X2_uniform.pdf}
\end{figure}

\subsection{Normalisation}

For $X$ of shape $N \times d$, we first demean the $N$ row-vectors to all have row-sum 0, then normalise so that they all have norm 1. 
As shown in Figure \ref{fig:nu_over_d_Xnormalised_nu_over_d_X_pareto.pdf}, for $N, d = O(1000)$, normalisation slightly increases Patnaik-Pearson dimension, with the effect being most pronounced for small dimension.

\begin{figure}
    \centering
    \includegraphics[width=0.45\textwidth]{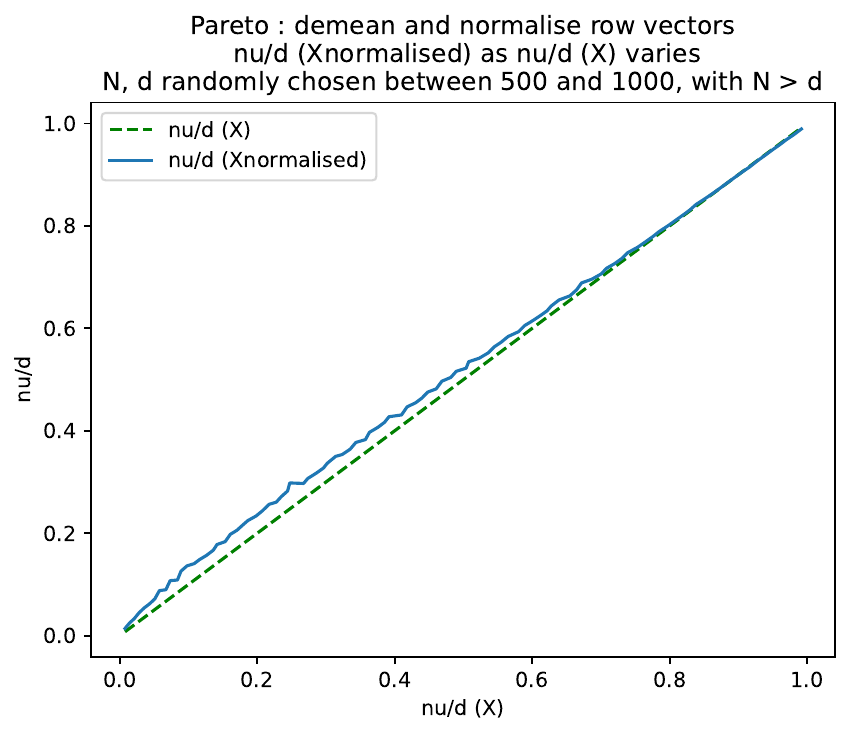}
    \includegraphics[width=0.45\textwidth]{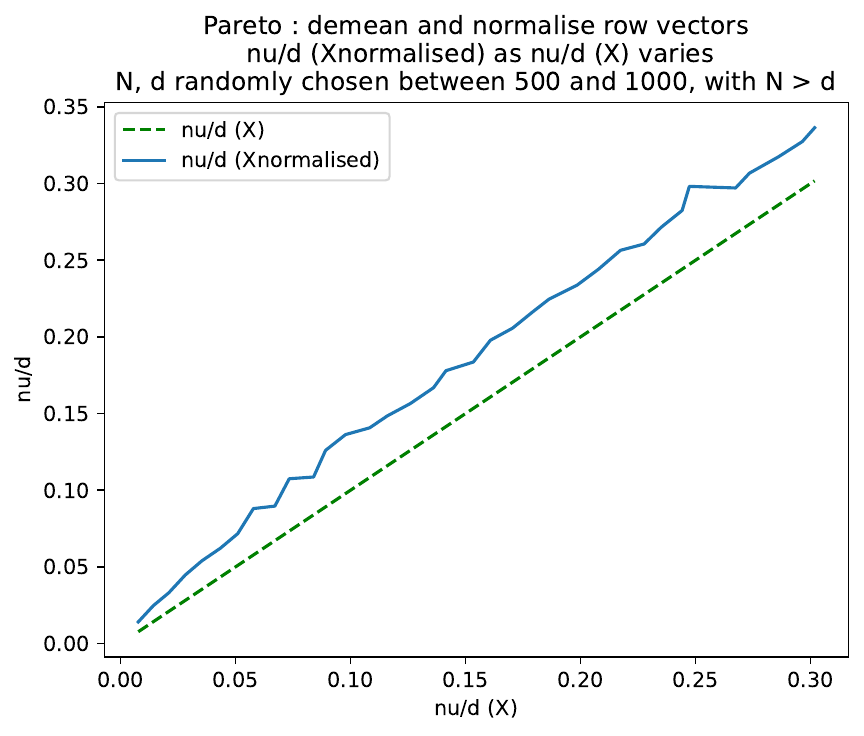}
    \caption{Normalisation : Patnaik-Pearson dimension for 
	$X_{\normalised}$ 
	as $\PatnaikPearson(X)$ varies.} 
    \label{fig:nu_over_d_Xnormalised_nu_over_d_X_pareto.pdf}
\end{figure}

\section{Application: BERT and DeepSeek}

\subsection{BERT embeddings}

The BERT model \cite{bert}, released in 2018, was the first open-weighted state of the art LLM.
We work with the BERT base model, which, at 110 million parameters, can no longer really be considered a ``large" language model. 
It is well known \cite{ethayarajh_how_contextual, gao_representation_degeneration, li_sentence_embeddings}
that for various LLMs (including BERT) word and sentence embeddings suffer from anisotropy - they are not uniformly distributed with respect to direction, instead occupying a narrow cone in the vector space.
This leads to difficulty in using such embeddings via similarity metrics such as cosine similarity and dot products.
This can be addressed by techniques such as whitening \cite{su_whitening}.
Li et al. \cite{li_sentence_embeddings} found that the BERT sentence embedding space suffers from two problems - word frequency biases the embedding space and low-frequency words disperse sparsely. In this section we are looking at token embeddings rather than sentence embeddings.

BERT is an encoder-only model, consisting of 12 layers, with a hidden size of 768 and context length of 512 tokens. 
The token vocabulary is 30,522, of which 29,528 tokens are valid, and 994 invalid.
Specifically, out of the first 999 tokens, five are special tokens - [PAD], [UNK], [CLS], [SEP] and [MASK] - which we keep, and the remaining 994 are [unusedN], which we seperate out as invalid. The remaining 29,523 tokens are considered to be valid.
We present results of three numerical experiments:

\begin{enumerate}
\item{{\bf Anisotropy : Common components of token embeddings.}
First, for the 994 invalid embeddings, the norms average 1.161, with standard deviation 0.001. Averaging the embeddings gives a common component of norm 1.154, and residualising this out reduces the average norm of the residualised embeddings to 0.125, with standard deviation 0.003. This residualisation has very little effect on the Patnaik-Pearson dimension, both pre- and post-residualisation this is 561. 

For the 29,528 valid embeddings, the norms average 1.410, with standard deviation 0.191. The common component has norm 0.937, and after residualisation the average norm is 1.176, with standard deviation 0.192. 
}
\item{{\bf Patnaik-Pearson dimension of the collection of valid token embeddings.}
We calculate the Patnaik-Pearson dimension of various collections $X$ of token embeddings, realised as an $N \times d$ matrix. Here $d = 768$, but it is debatable what the most natural value of $N$ should be. 
The full data manifold of all 29,528 valid token embeddings is computationally inaccessible. 
The context length of the BERT-base model is 512, 
which suggests $N = 512$ is a natural value to consider, although this bounds $\PatnaikPearson(X) \leq 512$ for any such $X$. 
We sample $N$ tokens without replacement from the full vocabulary of valid tokens, for $N$ between 100 and 10,000. Over this range the mean (over multiple samples) of ${\tfrac{1}{d}} \PatnaikPearson(X)$ increases from 0.11 to 0.85. See Figure \ref{fig:pp_dim_for_sampled_BERT_embeddings.pdf}.

An alternative approach, which we have not pursued, but might be considered to be semantically more meaningful, would be to select samples of genuine text from different sources, tokenize these, and calculate the Patnaik-Pearson dimension of the resulting data manifolds.  

\begin{figure}[htbp]
    \centering
    \includegraphics[width=0.45\textwidth]{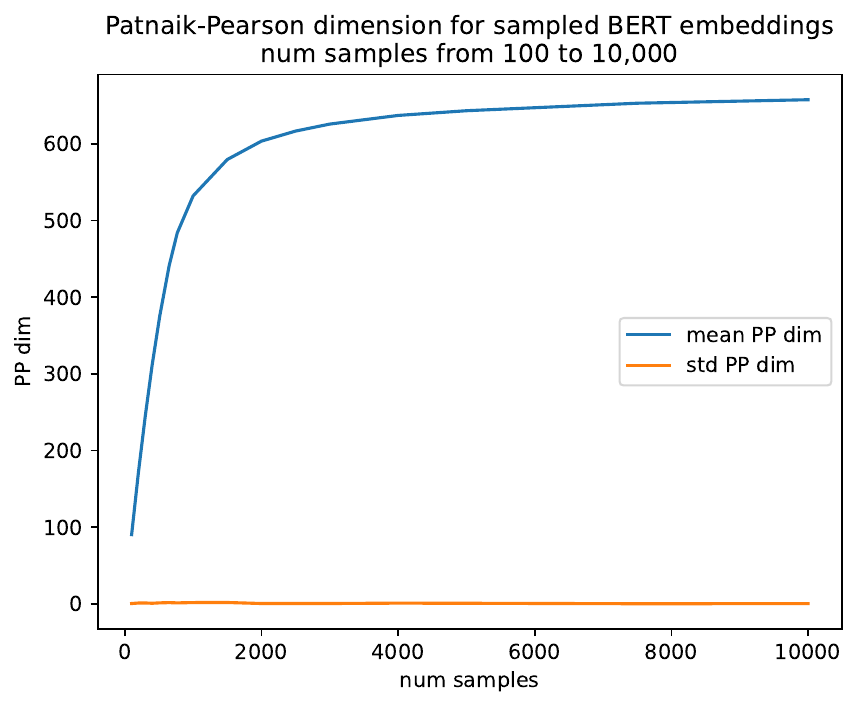}
    \includegraphics[width=0.45\textwidth]{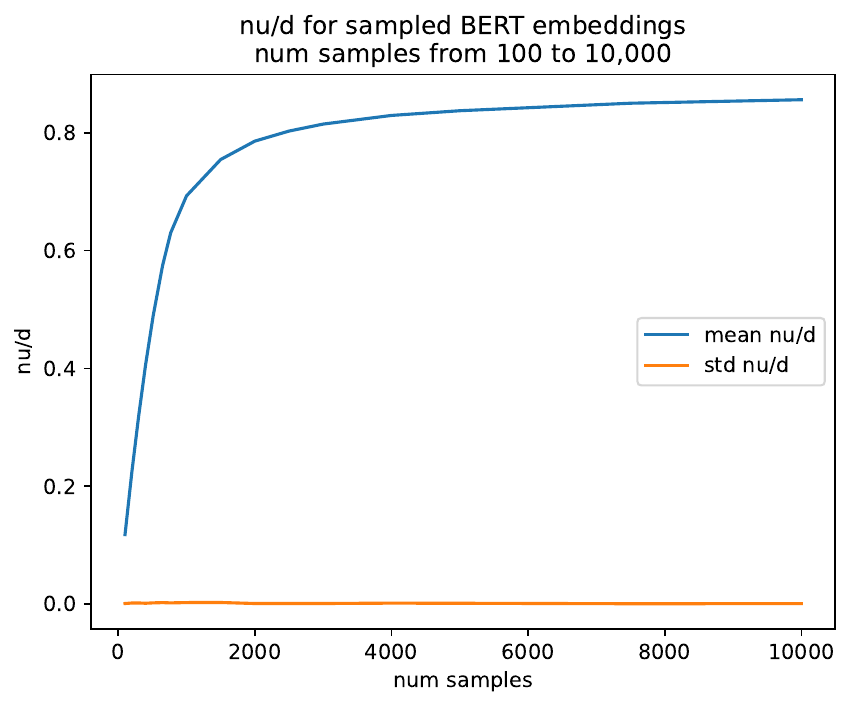}
    \caption{Patnaik-Pearson dimension and nu/d for sampled BERT token embeddings} 
    \label{fig:pp_dim_for_sampled_BERT_embeddings.pdf}
\end{figure}
}
\item
{
{\bf Layerwise evolution of the Patnaik-Pearson dimension.}
We repeatedly sample (without replacement) 512 embeddings from the full set of token embeddings. 
We calculate the Patnaik-Pearson dimension of the resulting data manifold as it passes through the layers of the model. 
As shown in Figure \ref{fig:bert_base_layer_pp_dim.pdf}, dimension decreases steadily from an initial value of approximately 380 to approximately 120 at the final layer.  

\begin{figure}[htbp]
    \centering
    \includegraphics[width=0.45\textwidth]{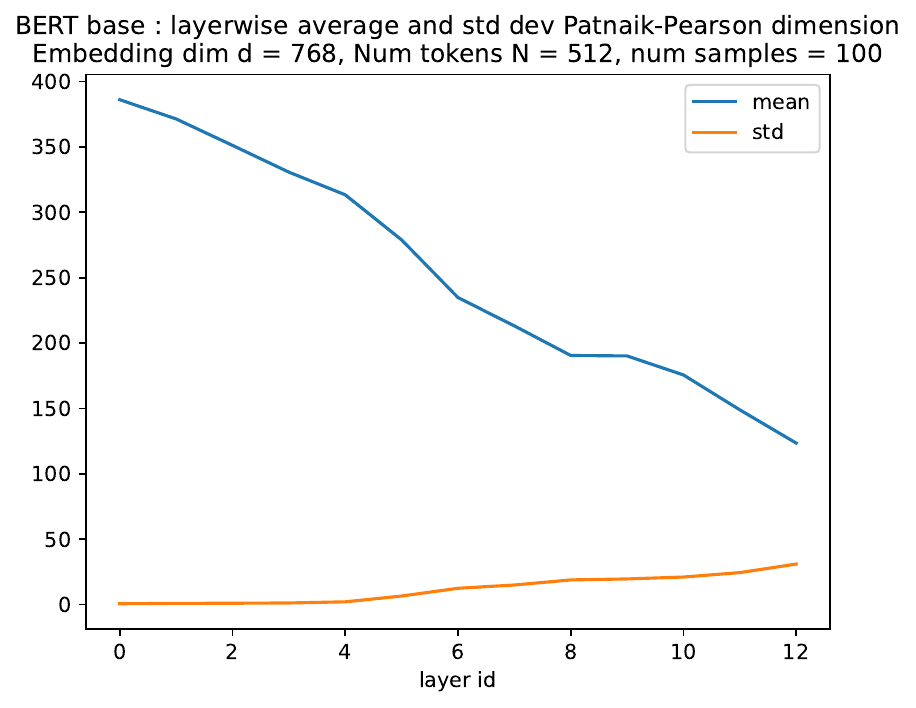}
    \includegraphics[width=0.45\textwidth]{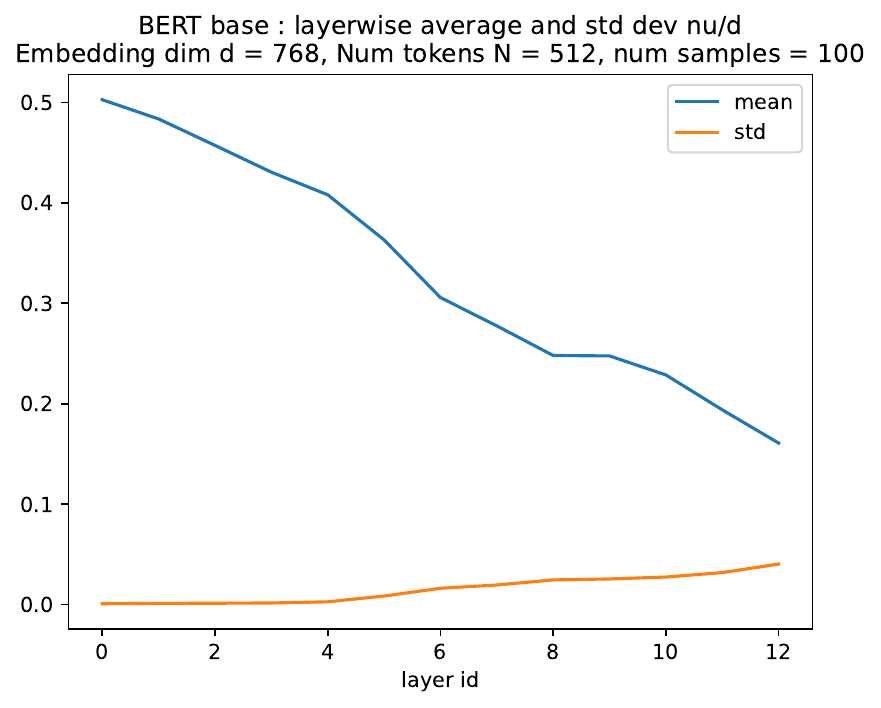}
    \caption{Layerwise evolution of Patnaik-Pearson dimension for BERT.} 
    \label{fig:bert_base_layer_pp_dim.pdf}
	\label{fig:bert_base_layer_nu_over_d.pdf}
\end{figure}

}

\end{enumerate}

\subsection{DeepSeek embeddings}

We perform the same experiments for a small version of the DeepSeek-R1 model \cite{deepseek}, namely DeepSeek-R1-Distill-Qwen-1.5B.
This is a 1.5B parameter generative (decoder-only) model. It has 28 layers, compared to BERT's 12, the hidden size is 1536, not 768, and the context length is roughly 130k, compared to 512 for BERT. The token vocabulary size is 151,643 tokens.

\begin{enumerate}
\item{ {\bf Patnaik-Pearson dimension of the collection of token embeddings.}
Since the context length is comparable to the size of the vocabulary, and
dealing with a data manifold of dimension 151,643 $\times$ 1536 is computationally inaccessible to us, we sample (without replacement) from the full set of embeddings, and calculate the Patnaik-Pearson dimension of this collection.  
Figure \ref{fig:deepseek_token_embedding_nu_over_d.pdf} shows the results of sampling collections of $N$ tokens (for $N$ ranging between 500 and 20,000) and calculating the average Patnaik-Pearson dimension across all sampled collections of size $N$.  We can see that the average value of ${\tfrac{\nu}{d}} = {\tfrac{1}{d}} \PatnaikPearson (X)$ increases from roughly 0.5 for $N = 1536$ to more than 0.9 for $N = 20,000$. We conjecture that as $N$ approaches the total vocabulary size, $\PatnaikPearson(X)$ approaches 1536, i.e. the token embeddings exploit the full dimensionality of the embedding space.

\begin{figure}[h]
    \centering
    \includegraphics[width=0.45\textwidth]{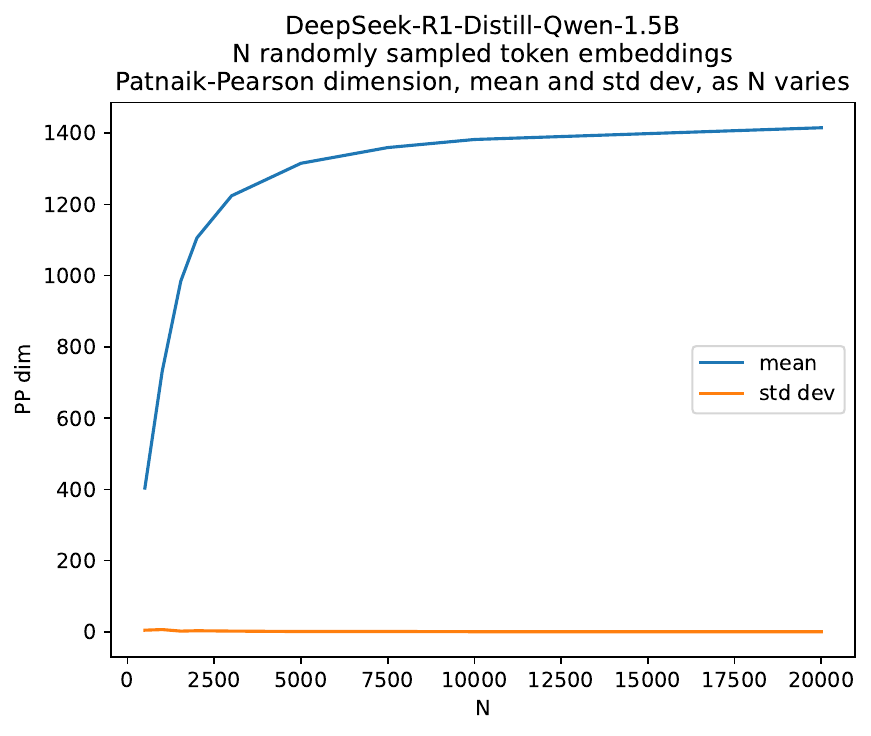}
    \includegraphics[width=0.45\textwidth]{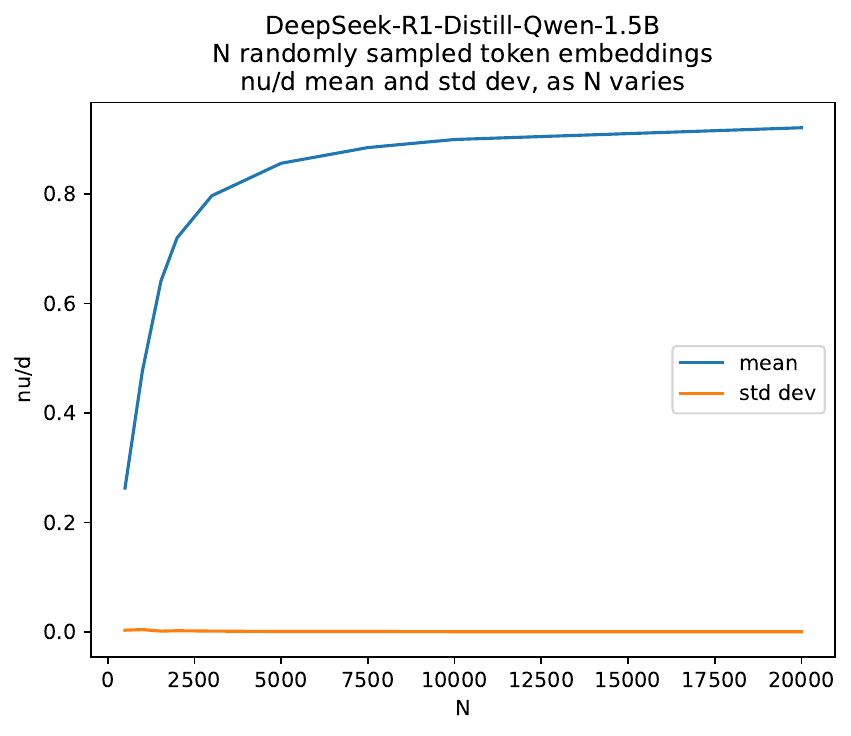}
    \caption{Deepseek token embeddings: Patnaik-Pearson dimension and nu/d for samples of token embeddings of size between 500 and 20,000.} 
    \label{fig:deepseek_token_embedding_nu_over_d.pdf}
\end{figure}
}
\item{{\bf Layerwise evolution of the Patnaik-Pearson dimension.}

We generate random samples of token embeddings (chosen without replacement), and track the evolution of their Patnaik-Pearson dimension as they pass through the layers of the model. 
For the results shown in Figure \ref{fig:deepseek_layerwise_pp_dim.pdf} we generated 20 random samples of 2000 token embeddings each (so $N = 2000$, $d = 1536$). 
As shown, the initial Patnaik-Pearson dimension is on average around 1200, but is compressed dramatically to around 200 in the early layers of the model, and then remains consistently at that level until about layer 20, after which point it grows again to around 500. 
}
\end{enumerate}

\begin{figure}[h]
    \centering
    \includegraphics[width=0.45\textwidth]{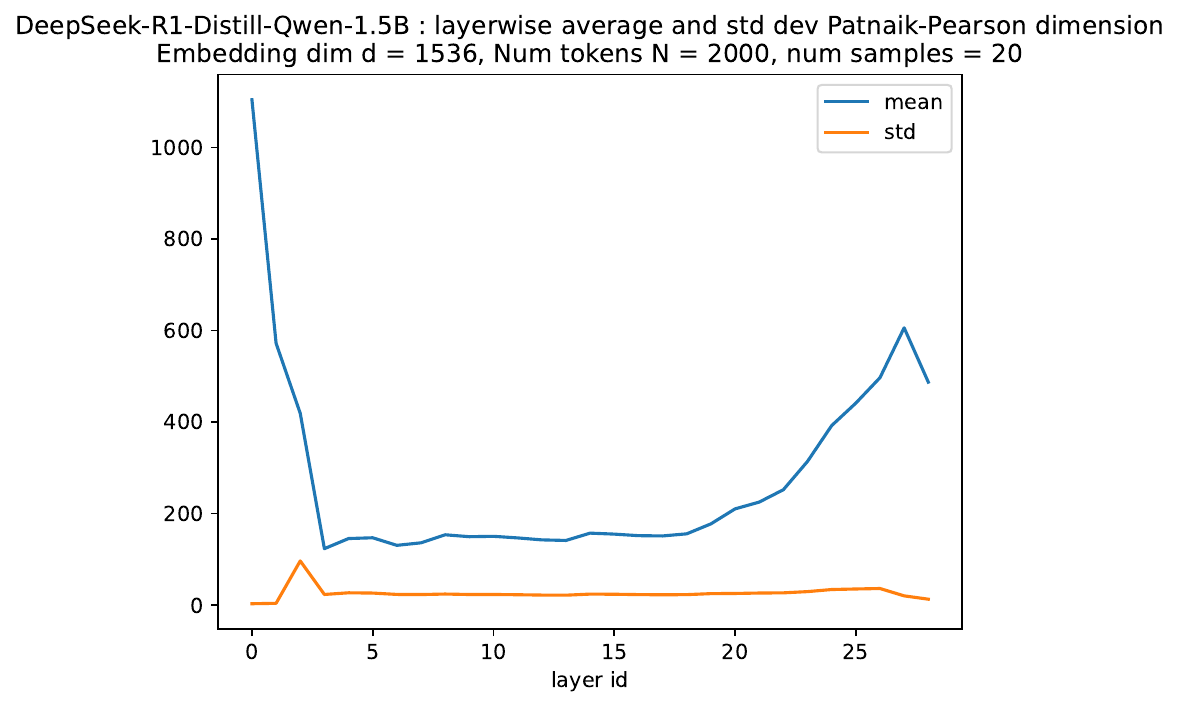}
    \includegraphics[width=0.45\textwidth]{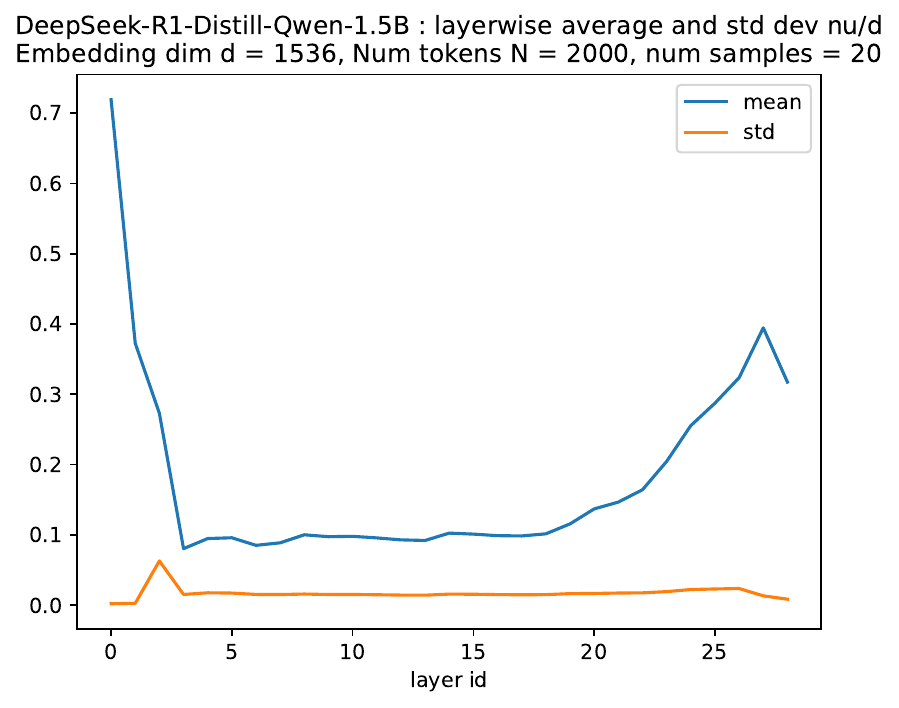}
    \caption{Layerwise evolution of the Patnaik-Pearson dimension of Deepseek embeddings.} 
    \label{fig:deepseek_layerwise_pp_dim.pdf}
	\label{fig:deepseek_layerwise_nu_over_d.pdf}
\end{figure}

\section{Conclusion}

In this work we have defined a novel measure of intrinsic dimension of a data manifold, which we call the Patnaik-Pearson dimension, due to its connection to Patnaik and Pearson's moment-matching formulae.
This was motivated by the application of the TwoNN intrinsic dimension estimator to a simple generative model of a data manifold. 
We investigate the properties of the Patnaik-Pearson dimension, and show that it has a close relation to the phenomena observed by Martin, Mahoney and Hinrichs for heavy-tailed weight matrices in neural networks, which motivated the development of their HTSR and SETOL theories. 
In particular, the critical values for the tail exponent for both the Patnaik-Pearson dimension and HTSR and SETOL coincide.
Using a combination of theoretical and numerical techniques, we study the behaviour of the Patnaik-Pearson dimension of a data manifold under the transformations typical to neural networks - multiplication by weight matrices; application of activation functions and softmax; addition, interpolation and concatenation; layer normalisation; attention. Some of these operations typically decrease Patnaik-Pearson dimension, others tend to increase it. 
It remains an open question as to in what generality these results hold. 
We apply this theory to two examples - the BERT-base and DeepSeek-R1-Distill-Qwen-1 models - and observe a pattern of Patnaik-Pearson dimension decreasing as the data manifold passes through the model layers.   

\section{Acknowledgements}

I am grateful to Anthony Coache, Eyal Neumann, Philipp Jettkant, Yifan Jiang, Alessandra Luati, Charles Martin, Antoine Jacquier, Jingjie Zhang and Dimitri Shlyakhtenko for very useful comments.

\section{Appendix : The TwoNN intrinsic dimension formula}
\label{appendix_twonn}

We derive the formulae from \cite{twonn} used in Section \ref{subsection_twonn}, in particular (\ref{mu_pdf_cdf}).
Suppose we have a region in $\RR^d$, containing random points with (uniform) density $\rho$.
This means, given an infinitesimal volume $\delta V$, the random variable $X(\delta V)$, defined to be the number of points occurring in the volume $\delta V$, 
satisfies 
$$\PP ( X(\delta V) = 0 ) = 1 - \rho \delta V, \quad \PP ( X(\delta V) = 1 ) = \rho \delta V, \quad \PP ( X(\delta V) = k ) = 0 \quad \forall \,\, k > 1$$
For a larger volume $V$, denoting by $X(V)$, defined to be the number of points occurring in $V$, then:\\

\begin{theorem}
\label{thm_p_xv_e_xv}
For any $k \geq 0$, $\PP ( X(V) = k ) = {\frac{(\rho V)^k}{k!}}e^{-\rho V}$, and $\EE (X(V)) = \rho V$.
\end{theorem}
{\bf{Proof:}} Divide $V$ into $n$ equal volumes, for some large $n$. Then 
$$\PP(X({\tfrac{V}{n}}) = 0) = 1 - {\tfrac{\rho V}{n}}, \quad \PP(X({\tfrac{V}{n}}) = 1) = {\tfrac{\rho V}{n}}$$
For any $k \geq 0$,
$\PP(X(V) = k) = {\tfrac{n!}{k! \, (n-k)!}} ({\tfrac{\rho V}{n}})^k (1 - {\tfrac{\rho V}{n}})^{n-k} \quad \rto \quad {\tfrac{(\rho V)^k}{k!}}e^{-\rho V}$ 
as 
$n \rto \infty$.
Therefore
$$\EE(X(V)) = \sum_{k=0}^\infty k \PP(X(V) = k) 
= \sum_{k=0}^\infty k {\tfrac{(\rho V)^k}{k!}}e^{-\rho V} 
= \sum_{k=1}^\infty {\tfrac{(\rho V)^k}{(k-1)!}}e^{-\rho V} 
= \rho V \sum_{k=0}^\infty {\tfrac{(\rho V)^k}{k!}}e^{-\rho V} 
= \rho V \quad \myqed$$

Suppose we have a collection $X$ of $N$ points $\{ \xsubone , \ldots , \xsubN \}$ in $\RR^d$, assumed to lie on a submanifold of dimension $m$. We want to estimate $m$. 
Given $\xsubi$, let 
$r_{i,1} \leq  r_{i,2} \leq \ldots $ be a sorted list of the distances from $\xsubi$ to its nearest neighbors. 
So $r_{i,1}$ is the distance from $\xsubi$ to its nearest neighbour, $r_{i,2}$ the distance to the second-nearest neighbour, and so on.
By convention define $r_0 = 0$.
The volume of the $m$-dimensional shell enclosed between successive 
neighbors $l-1$ 
and $l$ is
${\Delta V}_l = \omega_m ( r_l^m - r_{l-1}^m)$
where 
$\omega_m = \frac{\pi^{m/2}}{\Gamma({\frac{m}{2}}+1)}$
is the volume of the $m$-sphere of radius 1. 
By Theorem \ref{thm_p_xv_e_xv},
$$\PP({\Delta V}_l \in [v, v + \delta v]) = \PP( {{\Delta V}_l} \geq v 
\, \cap \,
{{\Delta V}_l} \leq v + dv )
= \PP( X( B_{v + \delta v}) = 1 | X( B_v) = 0) \PP(X( B_v) = 0)$$
$$= \PP( X( B_{\delta v}) = 1) \PP(X( B_v) = 0) 
= (\rho \delta v)  (e^{-\rho v})  = \rho  e^{-\rho v} \delta v$$
where $B_v$ is a ball of volume $v$.
Now consider the two shells ${\Delta V}_1$ and ${\Delta V}_2$. 
Define $R = \frac{{\Delta V}_2}{{\Delta V}_1}$.
Then 
$$
\PP (R \in [ \Rbar, \Rbar + \delta \Rbar])
= {\int_0^\infty} d v_1 {\int_0^\infty} d v_2 \rho^2 e^{- \rho(v_1 + v_2)} {\bf{1}}_{\{ {\frac{v_2}{v_1}} \in [ \Rbar, \Rbar + \delta \Rbar] \}}
= {\int_0^\infty} d v_1 
\rho^2 e^{- \rho v_1} 
{\int_{\Rbar v_1}^{\Rbar v_1 + \delta \Rbar v_1 }} 
d v_2  e^{- \rho v_2}
$$
$$
= {\int_0^\infty} d v_1 \rho^2 e^{- \rho v_1} {\tfrac{1}{\rho}} e^{\rho \Rbar v_1} [ e^{- \rho \delta \Rbar  v_1} - 1]
= {\int_0^\infty} d v_1 \rho e^{- \rho (1 + \Rbar) v_1}  \rho \delta \Rbar v_1
= \rho^2 \delta \Rbar {\int_0^\infty} d v_1 \, v_1 e^{- \rho (1 + \Rbar) v_1} 
$$\\
Now, using  
$\int_0^\infty dx \, x e^{\lambda x} = {\frac{1}{\lambda^2}}$, (provided $\lambda < 0$) this becomes
$
\PP (R \in [ \Rbar, \Rbar + \delta \Rbar]) 
= {\frac{\delta \Rbar}{(1 + \Rbar)^2}}
$.
Hence the pdf for $R$ is given by $f_R(r) = {\frac{1}{(1 + r)^2}}$.
To estimate $m$, define $\mu = {\frac{r_2}{r_1}} \in [1, \infty)$. Then
$$
R = {\frac{\Delta v_2}{\Delta v_1}} 
= {\frac{ \omega_m (r_2^m - r_1^m)}{\omega_m (r_1^m - r_0^m)}} 
= \mu^m - 1
$$
Therefore, the pdf of $\mu$, for $\mu \geq 1$, is given by
$$
f_{\mu}(t) = {\frac{1}{\delta t}} \PP ( \mu \in [t, t + \delta t]) 
= {\frac{1}{\delta t}} \PP ( 1 + R \in [ t^m , (t + \delta t)^m ])
= {\frac{1}{\delta t}} {\frac{1}{t^{2m}}} m t^{m-1} \delta t
= m t^{-(m+1)}
$$
Hence (see \cite{twonn}, Equation (5)), the pdf and CDF are given by 
\begin{equation}
\label{pdf_mu}
f_{\mu} (t) = m t^{-(m+1)} {\bf{1}}_{[1,\infty)},
\quad
F_{\mu}(x) = (1 - x^{-m}) {\bf{1}}_{[1,\infty)}
\end{equation} 
This establishes (\ref{mu_pdf_cdf}). 

\end{document}